\documentclass[a4paper,11pt]{article}
\usepackage[latin1]{inputenc}
\usepackage[T1]{fontenc}
\usepackage[british,UKenglish,USenglish,american]{babel}
\usepackage{amsmath}
\usepackage{amsfonts}
\usepackage{hyperref}
\usepackage[T1]{fontenc}
\usepackage[latin1]{inputenc}
\usepackage{theorem}
\usepackage{graphics}
\usepackage{makeidx}
\usepackage{fancyhdr}
\usepackage{bbm}
\usepackage{amssymb}
\usepackage{color}
\numberwithin{equation}{section}
\newcommand{\be}{\begin{equation}}
\newcommand{\ee}{\end{equation}}
\newcommand{\benn}{\begin{equation*}}
\newcommand{\eenn}{\end{equation*}}
\newcommand{\bea}{\begin{eqnarray}}
\newcommand{\eea}{\end{eqnarray}}
\newcommand{\beann}{\begin{eqnarray*}}
\newcommand{\eeann}{\end{eqnarray*}}

\newtheorem{theorem}{Theorem}[section]
\newtheorem{proposition}[theorem]{Proposition}

\newtheorem{lemma}[theorem]{Lemma}

\newtheorem{definition}[theorem]{Definition}
\theorembodyfont{\rm}

\newtheorem{remark}[theorem]{Remark}
\newtheorem{example}[theorem]{Example}

\newcommand{\qed}{\hfill $\Box$\smallskip}

\def\R{\mathbb{R}}

\def\cA{\mathcal{A}}
\def\cB{\mathcal{B}}
\def\cC{\mathcal{C}}

\def\cK{\mathcal{K}}
\def\cL{\mathcal{L}}

\def\cR{\mathcal{R}}

\def\cU{\mathcal{U}}

\def\txtd{{\textnormal{d}}}
\def\txte{{\textnormal{e}}}

\def\Id{{\textnormal{Id}}}

\def\ra{\rightarrow}

\title{Pathwise mild solutions for quasilinear stochastic partial differential equations}

\author{Christian Kuehn\thanks{Technical University of Munich (TUM), 
Faculty of Mathematics, 85748 Garching bei M\"unchen, Germany}~~and~Alexandra 
Neam\c tu\thanks{Technical University of Munich (TUM), 
Faculty of Mathematics, 85748 Garching bei M\"unchen, Germany}}

\begin{document}

\maketitle

\begin{abstract}
Stochastic partial differential equations (SPDEs) have become a key modelling 
tool in applications. Yet, there are many classes of SPDEs, where the existence 
and regularity theory for solutions is not completely developed. Here we contribute 
to this aspect and prove the existence of mild solutions for a broad class of quasilinear 
Cauchy problems, including - among others - cross-diffusion systems as a key application. 
Our solutions are local-in-time and are derived via a fixed point argument in 
suitable function spaces. The key idea is to combine the theory of 
deterministic quasilinear parabolic partial differential equations (PDEs) with 
recent theory of evolution semigroups. We also show, how to apply our theory to 
the Shigesada-Kawasaki-Teramoto (SKT) model. Furthermore, we provide examples of 
blow-up and ill-posed operators, which can occur after finite-time. 
\end{abstract}
\section{Introduction}
\label{sec:intro}

In this work, we study SPDEs as abstract quasilinear Cauchy problems 
\begin{equation}
\label{introduction:equation}
\begin{cases}
~\txtd u(t)= [A(u(t))u(t) + F(t,u(t))]~\txtd t  + \sigma(t,u(t))~\txtd W(t)
,\hspace*{3 mm} t\in(0,\infty),~u(t)\in\R^d,\\
u(0)=u_{0},
\end{cases}
\end{equation}
where the precise assumptions on the coefficients are stated in 
Section~\ref{qspde}. One particular motivation is the SKT cross-diffusion
model~\cite{ShigesadaKawasakiTeramoto} with $d=2$, 
$A(u)=(\Delta (p_1(u)),\Delta (p_2(u)))^\top$ for quadratic polynomials $p_{1,2}$
in $u=(u_1,u_2)^\top$, $\Delta$ denoting the Laplacian, and $F$ also being a quadratic 
polynomial. The deterministic SKT system (i.e., $\sigma\equiv 0$) and its variants have 
been studied very intensively~\cite{Amann1,ChenJuengel2,LouNi,LouNiWu,Ni}.
Furthermore, there are many other motivations as the form of the deterministic
part (or drift terms) $A$ and $F$ encompasses a much wider class of 
PDEs~\cite{Amann4,Yagi1}. For many applications, it is very important to
consider noise terms ($\sigma\not\equiv 0$) due to intrinsic finite system-size
noise or external fluctuations acting on the system.\medskip

Here we aim to develop a theory for quasilinear stochastic evolution 
equations~\eqref{introduction:equation} using a semigroup approach. The main
theme is to extend the very general deterministic theory of quasilinear
Cauchy problems~\cite{Amann2,Yagi,Yagi1}. The key idea is to employ a modified 
definition of mild solutions~\cite{PronkVeraar} for the quasilinear case in 
comparison to the more classical parabolic SPDE setting~\cite{DaPratoZabczyk}. 
Before we describe our approach in more detail, we briefly review 
some other techniques and solution concepts used for (certain subclasses of) the 
SPDE~\eqref{introduction:equation}. Instead of mild solutions, one may instead
use weak, or martingale, solutions~\cite{BertiniButtaPisante,DebusscheHofmanovaVovelle,
DenisStoica,HofmanovaZhang} of~\eqref{introduction:equation}; here weak solution 
is interpreted in the classical PDE sense while these solutions are also sometimes 
referred to as strong solutions from a probabilistic 
perspective~\cite{PrevotRoeckner}. There are also several
works exploiting the additional assumption of monotone 
coefficients~\cite{LiuRoecknerdaSilva} particularly in the case of the stochastic 
porous medium equation~\cite{BarbuDaPratoRoeckner,BarbuDaPratoRoeckner1}, where 
$A(u)=\Delta(a(u))$ for a maximal monotone map $a$ and $F\equiv 0$. Other
approaches to quasilinear SPDEs are based upon a gradient structure~\cite{Gess}, 
approximation methods~\cite{Kim1,KobayasiNoboriguchi}, kinetic 
solutions~\cite{DebusscheHofmanovaVovelle,GessHofmanova}, or directly 
looking at strong (in the PDE sense) solutions~\cite{Hornung1}.\medskip 

One may ask, why one might want to prove the existence of pathwise mild solutions 
obtained by a suitable variations-of-constants/Duhamel formula~\cite{Henry} instead 
of working with weak solutions obtained in a formulation via test 
functions~\cite{Evans}? One reason is that a mild formulation is often more natural 
to work with in the context of (random) dynamical systems for 
SPDEs~\cite{CrauelFlandoli,Henry}. In fact, many classical results regarding dynamics and long-time behavior of semilinear SPDEs 
are often crucially based upon the mild formulation and semigroups~\cite{Henry}.  
We expect this theory to generalize a lot easier also in the quasilinear case if one does not have to work with
weak(er) solutions. If we take
the SKT system again as a motivation, then there are deterministic results regarding
the existence of attractors using weak~\cite{PhamTemam} as well as 
mild~\cite{Yagi2} solutions concepts. We intend to investigate the existence of random attractors for the stochastic SKT equation using the mild formulation in a future work, since it perfectly fits into the framework of random dynamical systems. A second reason to consider mild solutions is
that it should be easier to derive space-time regularity~\cite{DebusscheDeMoorHofmanova} 
of the solution for~\eqref{introduction:equation}. Estimates for the nonlinear terms 
also tend to simplify in a mild solution setting already for SODEs~\cite{BerglundGentz}. 
A third reason to consider mild solutions is that they are more natural in the setting
of regularity structures~\cite{Hairer1,Hairer3}, \cite{GerecserHairer}, where convolution with the heat kernel is
a key tool; in this context generalizations of regularity structures to 
quasilinear SPDEs turn out to be very subtle~\cite{BerglundKuehn}, \cite{GerecserHairer}. So a 
better understanding for more regular noises should be helpful. To deal with 
rough noises an alternative to regularity structures is to stay closer to a 
paracontrolled approach~\cite{GubinelliTindel}, which has recently been proposed 
for certain quasilinear SPDEs~\cite{BailleulDebusscheHofmanova,FurlanGubinelli,
OttoWeber,OttoWeber1}.\medskip

Since the linear operator $A$ depends on the solution itself, which will be in our 
case a stochastic process, we cannot apply the standard fixed-point argument 
as in ~\cite{Amann2,Yagi1}. Namely, if we denote with $U^{u}$
the random evolution operator generated by $A(u)$, one naturally expects that the 
mild solution of~\eqref{introduction:equation} should be given by the 
variation-of-constants formula
\begin{equation}
\label{conv1}
u(t) = U^{u}(t,0) u_{0} + \int\limits_{0}^{t} U^{u}(t,s)F(s,u(s))~\txtd s 
+ \int\limits_{0}^{t} U^{u}(t,s)\sigma(s,u(s))~\txtd W(s).
\end{equation}
As already observed in~\cite{PronkVeraar}, and justified in Sections~\ref{p} 
and~\ref{qspde}, the random evolution operator $U^{u}(t,s,\omega)$ does not 
satisfy the necessary adaptedness properties required in order to define the 
It\^{o}-integral. More precisely, it turns out to be only 
$\mathcal{F}_{t}$-adapted and not $\mathcal{F}_{s}$-adapted. Consequently, 
the stochastic convolution given in~\eqref{conv1} is not well-defined in the 
It\^{o}-sense. This situation is commonly met for instance in the theory of 
stochastic evolution equations with time-dependent random 
generators~\cite{PronkVeraar}. A way out of this situation is to introduce 
a new concept of mild solution for~\eqref{introduction:equation}, which is 
based on the integration-by-parts formula for stochastic convolutions. This 
is motivated in \cite[Sec.~4]{PronkVeraar} as well as in Appendix~\ref{a} here
for convenience. Using this approach, we prove by means of fixed-point arguments 
that the mild solution of~\eqref{introduction:equation} is given by
\begin{align}\label{pms}
u(t) & = U^{u}(t,0) u_{0} + U^{u}(t,0) \int\limits_{0}^{t}
 \sigma(r,u(r))~\txtd W(r) + \int\limits_{0}^{t} U^{u}(t,s) F(s,u(s))~\txtd s
\nonumber \\
& -\int\limits_{0}^{t} U^{u}(t,s) A(u(s)) \int\limits_{s}^{t} 
\sigma(r,u(r)) ~\txtd W(r)~\txtd s.
\end{align}
One can show that the stochastic convolutions appearing in the formula above 
can be defined pathwise, therefore we will call this a pathwise mild solution 
of~\eqref{introduction:equation}; see Section~\ref{qspde}.\\

 Additionally to the adaptedness issue stated above, note that there are several technical difficulties required in order to obtain sufficient space-time regularity results for (\ref{pms}) which are necessary to set-up the fixed-point argument from the deterministic theory of \cite{Yagi}. Therefore, it is by no means straightforward to see why (\ref{pms}) is the right solution concept for our original problem and how these theories fit together. We also emphasize that up to now there is no theory available for mild solutions for quasilinear SPDEs in contrast to the deterministic case. The semigroup approach has turned out to be a very powerful tool for the analysis of quasilinear PDEs, see \cite{pazy}, \cite{Amann1}, \cite{Yagi} and the references specified therein.  
 This work represents an important step in exploiting the semigroup methods from the deterministic setting in the stochastic one.\\

Another important 
feature of this approach is that it can be applied to more general stochastic 
processes $(S(t))_{t\in[0,T]}$ not just the Brownian motion, since one
needs to establish an integration theory only for 
\benn
\int\limits_{0}^{T} \sigma(t)~\txtd S(t) 
\eenn
under suitable assumptions on $\sigma$. In this case $S$ does not even have to 
be a semimartingale, so one can consider~\eqref{introduction:equation} perturbed 
by an additive-fractional noise thereby generalizing results 
in~\cite{LiuTudor,LiuYan}; we intend to explore this in a future work.\\ %%Finally, using 
%%the mild formulation one can also consider stochastic quasilinear delay equations. We also intend to investigate this aspect motivated by the applications of delay equations in mathematical biology/population dynamics.\\

We deal here with local in time existence of solutions for (\ref{introduction:equation}). To obtain global-in-time solutions for~\eqref{introduction:equation}, one has to 
solve two further issues: (a) finite-time blow-up and (b) degeneration of the operator
$A(u)$. The latter is often related in practice to preserving certain positivity
assumptions present in the initial data and naturally relates to stochastic maximum 
principles~\cite{BarbuDaPratoRoeckner,DenisMatoussi,DenisMatoussiStoica,Oksendal1}. 
The issue (a) of blow-up certainly also occurs already for many classical 
SPDEs~(see e.g.~\cite{deBouardDebussche,DozziLopezMimbela,MuellerSowers1}) but 
also plays a key role for quasilinear SPDE problems~\cite{Hornung1}. For completeness, 
we provide two very simple quasilinear counter-examples involving (a) and (b) 
to demonstrate that we cannot expect global-in-time existence 
for~\eqref{introduction:equation} in general; see Section~\ref{examples}. 
Yet, we conjecture that for many quasilinear SPDEs, where
global existence is known for the PDE ($\sigma\equiv 0$), there is a natural
choice of noise term $\sigma\not\equiv 0$ such that also the SPDE has global-in-time
existence.  The choice of noise term is definitely case-dependent but it should be
possible to deal with many cases arising directly from modelling considerations,
which is another interesting direction for future research.
%%%%%%%%%%%%%%%%%%%%%%%%%%%%%%%%%%%%%%%%%%%%%%%%%%%%%%%%%%%%%%%%%%%%%%%%%%%%%%%%% 
\section{Preliminaries}
\label{p}

Throughout this work we fix a time horizon $T>0$ and a stochastic basis $(\Omega,\mathcal{F}, (\mathcal{F}_{t})_{t\geq 0}, \mathbb{P})$ with a complete, right-continuous filtration. We also use the standard notation $\omega$ for elements in $\Omega$. \\

We first collect basic properties and results concerning parabolic evolution families 
generated by random nonautonomous operators $\{A(t,\omega)\}_{t\in[0,T]}$ on a separable 
Banach space $X$. We make following assumptions, which are going to ensure that $A(t,\omega)$ 
generates a \emph{parabolic evolution system}, which is a family of linear operators 
depending on two-time parameters and having similar properties to analytic 
$C_{0}$-semigroups \cite[Section~II.2]{Amann4}:
	
\begin{itemize}
	\item [(A1)] The spectrum of $A(t,\omega)$ is contained in an open sectorial domain, 
	more precisely
	$$\sigma(A(t,\omega))\subset \Sigma_{\varphi}:=\{ \lambda\in\mathbb{C}\mbox{ : } 
	|\mbox{arg}\lambda|<\varphi \}, \mbox{ for } (t,\omega)\in[0,T]\times\Omega $$
	with a fixed angle $0<\varphi<\frac{\pi}{2}$.
	\item [(A2)] There exists a constant $M\geq 1$ such that the resolvent estimate
	$$||(\lambda \Id -A(t,\omega))^{-1} ||_{\mathcal{L}(X)}\leq \frac{M}{|\lambda |+1}$$ 
	holds true $\mbox{ for } \lambda\notin \Sigma_{\varphi} \mbox{ and }
	(t,\omega)\in[0,T]\times\Omega$, where $\mathcal{L}(X)$ is the space of linear 
	operators on $X$ with norm $\|\cdot\|_{\mathcal{L}(X)}$ 
	and $\Id$ is the identity map.
	\item [(A3)] The Acquistapace-Terreni condition~\cite{Acquistapace} is fulfilled, 
	namely there exist two exponents $\nu,\delta\in(0,1]$ with $\nu+\delta>1$ such 
	that for every $\omega\in\Omega$ there exists a constant $L(\omega)\geq 0$ such 
	that for all $s,t\in[0,T]$ we have
	\begin{equation}\label{at}
	|| A^{\nu}(t,\omega) (A(t,\omega)^{-1} - A(s,\omega)^{-1}) ||_{\mathcal{L}(X)}
	\leq L(\omega) |t-s|^{\delta}. 
	\end{equation}
\end{itemize} 

Let $\Gamma:=\{(s,t) \in[0,T]^{2} \mbox{ : } s\leq t \}$. In the following, for 
notational simplicity we drop the $\omega$-dependence of $A(t)$ and only indicate it explicitly for certain statements, where its crucial role is clarified.\\

\textbf{Covention}:  Since certain expressions and constants depend on several parameters, i.e.~$A(t,\omega)$ we always omit the last symbol whenever this dependence is clear.\\

By applying the 
results in~\cite{Acquistapace} pointwise in $\Omega$, see 
also~\cite[Theorem~2.2]{PronkVeraar}, we obtain:

\begin{theorem}
\label{po}
There exists a unique map $U:\Gamma\times\Omega\to\mathcal{L}(X)$ such that
		\begin{itemize}
			\item [(T1)] for all $t\in[0,T]$, $U(t,t)=\emph{Id}$;
			\item [(T2)] for all $r\leq s\leq t$, $U(t,s)U(s,r)=U(t,r)$;
			\item [(T3)] for every $\omega\in\Omega$, the map $U(\cdot,\cdot,\omega)$ 
			is strongly continuous;
			\item [(T4)] there exists a mapping $C:\Omega\to\mathbb{R}_{+}$, such 
			that for all $s\leq t$, one has $$|| U(t,s)||_{\mathcal{L}(X)}\leq C. $$
			\item [(T5)] for every $s<t$ it holds $\frac{\partial}{\partial t} 
			U(t,s)=A(t)U(t,s)$ and $\frac{\partial}{\partial s} U(t,s)=-U(t,s)A(s)$ 
			pointwise in $\Omega$. Moreover, there exists a mapping 
			$C:\Omega\to\mathbb{R}_{+}$ such that $$||A(t)U(t,s) ||\leq C (t-s)^{-1}. $$
		\end{itemize}
\end{theorem}
	
Consequently, if (A1)-(A3) hold true $\{A(t,\omega)\}_{t\in[0,T]}$ generates the evolution system/family 
$\{U(t,s,\omega)\}_{0\leq s\leq t\leq T}$. In our case, see Section~\ref{qspde} the evolution system $U^{u}(t,s,\omega)$ will additionally depend on the solution $u$ of a quasilinear SPDE.\medskip

In contrast to the deterministic setting all constants specified above depend on $\omega$ which causes several technical difficulties, more precisely $C$ depends in general on $L(\omega)$ and on $\delta$. For instance applying Theorem 4.4.1 in \cite{Amann4} one has
\begin{align}\label{est:u}
||U(t,s,\omega)||_{\mathcal{L}(X)}\leq \widetilde{C} e ^{\mu(\omega)(t-s)},~\mbox{for } (t,s)\in\Gamma,
\end{align}
where $\mu(\omega,\delta)=\mu(\omega)=\widetilde{C}(\delta) L(\omega)^{1/\delta}$. Here $\widetilde{C}$ stands for an arbitrary constant and $\widetilde{C}(\delta)$ indicates the dependence of $\widetilde{C}$ on the H\"older exponent from (\ref{at}).
 We point out following fact regarding this issue which is crucial for the computation in Section~\ref{qspde}.
\begin{remark}
	 \begin{itemize}
	 	\item [1)] One can  assume that the mapping $L:\Omega\to\mathbb{R}^{+}$ 
	 	introduced in~\eqref{at} is bounded in $\Omega$, analogously 
	 	to~\cite[Section 5.2]{PronkVeraar}. The general case can be treated by a 
	 	localization argument~\cite[Section 5.3]{PronkVeraar}, namely one introduces appropriate stopping times $(\tau_{n})_{n\in\mathbb{N}}$ and considers $A_{n}(t,\omega):=A(t\wedge\tau_{n}(\omega),\omega)$.
	 	\item [2)] In our case the generators will depend on the solution itself, so it is meaningful to control the solution process $u$ in order to make sure that the corresponding evolution operator $U^{u}$ is indeed well-defined. %% avoid degeneracy and so on
	 	 Consequently we deal with $A_{n}(u(t,\omega)):=A(u(t\wedge\tau_{n}(\omega),\omega))$ as stated in (A1')-(A3') in Section~\ref{qspde}. Therefore, all constants arising from the estimates involving $A(u)$ and $U^{u}$ will depend on $\tau_{n}(\omega)$, as precisely specified in Section \ref{qspde} below.
	 \end{itemize}

\end{remark}

The following estimates for analytic $C_{0}$-semigroups and parabolic evolution 
operators are essential for the computation in Section \ref{qspde}. These can be 
looked up in \cite[Section 2.6, p.~69]{pazy}, \cite[Section 8.1, p.~154]{Yagi1} 
and the references specified therein; note that they hold pointwise in $\omega\in\Omega$
similar to Theorem~\ref{po}.

\begin{theorem} %% \widetilde{\mu} here, \widehat{\mu} domains and \mu for f!!!
\label{estan}
There exists a mapping $C:\Omega\to\mathbb{R}_{+}$ such that for $0\leq s <t \leq T$ and $\widetilde{\alpha},\widetilde{\beta}\in(0,1]$ we have 
	\begin{equation}\label{an}
	||A^{\widetilde\beta}(t) U(t,s) ||_{\mathcal{L}(X)} \leq C (t-s)^{-\widetilde\beta} \mbox{, } \quad 
	|| U(t,s)A^{\widetilde\beta}(t) ||\leq C (t-s)^{-\widetilde\beta},
	\end{equation}
	as well as
	\begin{equation} \label{an2}
	||A^{\widetilde\beta}(t) U(t,s)A^{-\widetilde\alpha}(s) ||_{\mathcal{L}(X)}\leq C (t-s)^{\widetilde\alpha-\widetilde\beta}.
	\end{equation} 
%%	Furthermore, for $\widetilde\theta\geq 0, \mbox{  } \widetilde\rho\in(\widetilde\theta,1+\widetilde\theta] \mbox{ and } 
%%	\widetilde\mu\in\mathbb{R}$
%%\begin{equation}\label{id}
%%||U(t,s) -\emph{Id} ||_{\mathcal{L}(X_{\widetilde\rho+\widetilde\mu}, X_{\widetilde\theta+\widetilde\mu})} 
%%\leq C (t-s)^{\widetilde\rho-\widetilde\theta}
%%\end{equation}	
%%holds true.
\end{theorem}

Again, $C$ depends in general on $L(\omega)$ and on $\delta$ introduced in (\ref{at}). For more details and properties of fractional powers of sectorial operators 
$A^{\widetilde\beta}$ for $\widetilde\beta>0$ and the usual fractional space $X_{(\cdot)}$ we refer the reader to~\cite[Theorem 6.13, p.~74]{pazy}.

\begin{remark}
	Note that the first two assumptions imply that $A(t)$ generates an analytic $C_{0}$-semigroup which is denoted by $e^{-r A(t)}$. In this case there exists a mapping $C:\Omega\to \mathbb{R}_{+}$ such that the estimate
	\begin{align*}
	||A^{\widetilde{\theta}}(t) e^{-r A(t)}||_{\mathcal{L}(X)} \leq C r^{-\widetilde{\theta}}, ~~\mbox{for } \widetilde{\theta}>0 \mbox{ and } r>0
	\end{align*}
	holds true.
\end{remark}

  Furthermore, we recall 
  the next result~\cite[Proposition~2.4]{PronkVeraar}, which deals with the 
  measurability of $U$:
  
  \begin{proposition}
  	\label{meas}
  	The evolution system $U:\Gamma\times\Omega\to\mathcal{L}(X)$ is strongly measurable 
  	in the uniform operator topology. Moreover, for each $t\geq s$, the mapping 
  	$\omega\mapsto U(t,s,\omega)\in\mathcal{L}(X)$ is strongly $\mathcal{F}_{t}$-measurable 
  	in the uniform operator topology. 	
  \end{proposition}
 
For the sake of completeness, we provide now some known results concerning 
stochastic calculus, which will be required further on. We let $H$ and $Z$ stand 
for two separable Hilbert spaces and $(W(t))_{t\in[0,T]}$ for an $H$-cylindrical 
Brownian motion, meaning that 
$$W(t)=\sum\limits_{n=1}^{\infty} w_{n}(t) e_{n}, $$ 
with $(w_{n}(\cdot{}))_{n\geq 1}$ being mutually independent real-valued standard 
Wiener processes relative to $(\mathcal{F}_{t})_{t\geq 0}$ and $(e_{n})_{n\geq 1}$ 
an orthonormal basis in the separable Hilbert space $H$. With $\mathcal{L}_{2}(H,Z)$ 
we denote the space of Hilbert-Schmidt operators from $H$ to $Z$.	
As justified in Section \ref{qspde}, for our aims it will be enough to analyze the 
stochastic integral respectively the stochastic process
\be 
\label{eq:intmidfirst}
\left(\int\limits_{0}^{t}\sigma(r)~\txtd W(r)\right)_{t\in[0,T]} 
\ee
only for strongly-measurable, adapted stochastic processes 
$\sigma\in L^{0}(\Omega;L^{2}(0,T;\mathcal{L}_{2}(H,Z)))$. Here $L^0$ indicates measurability and $L^p$ is going to denote the usual Lebesgue spaces. \\

We recall (consult \cite[Section 4.1]{PronkVeraar} and the references specified 
therein) that the process~\eqref{eq:intmidfirst} exists and is pathwise 
continuous for $\sigma\in L^{0}(\Omega; L^{2}(0,T;\mathcal{L}_{2}(H,Z) ))$.
Moreover, one has the one-sided estimate
\begin{eqnarray}
|| J(\sigma)||_{L^{p}(\Omega; \cC([0,T];Z)) } 
\leq C ||\sigma ||_{L^{p}(\Omega; L^{2}(0,T;\mathcal{L}_{2}(H,Z) ) ) },
\end{eqnarray} 
where for $t\in[0,T]$ we set $J(\sigma)(t):=\int\limits_{0}^{t} 
\sigma(r)~\txtd W(r) $. %% and $\cC(\cdot,\cdot)$ indicates a space of continuous 
%maps.
\medskip

Furthermore, for $J(\sigma)$ the following regularity 
results~\cite[Proposition 4.4]{PronkVeraar1} are available and will be employed 
in Section~\ref{qspde}. 

\begin{proposition}
\label{reg}
Let $0<\widehat\alpha<1/2$, $p\in[2,\infty)$ and $\sigma$ be a strongly measurable 
adapted process, belonging to $L^{0}(\Omega; L^{p}(0,T;\mathcal{L}_{2}(H,Z)) )$. 
Then $J(\sigma)\in W^{\widehat\alpha,p}(0,T;Z)$ almost surely (a.s.), where 
$W^{\widehat\alpha,p}$ is the notation for the usual Sobolev spaces.	
\end{proposition}

Due to the embedding 
$$W^{\widehat\alpha,p}(0,T;Z)\hookrightarrow C^{\widehat\alpha-\frac{1}{p}}(0,T;Z)\quad \text{
for $\frac{1}{p}<\widehat\alpha<\frac{1}{2}$},$$ 
one obtains H\"older regularity of the integral process, 
namely $J(\sigma) \in \cC^{\widehat\alpha-\frac{1}{p}}(0,T;Z)$ a.s.. The next crucial result 
will be used throughout the next subsection, see \cite[Proposition 4.1]{PronkVeraar} 
for the full generality of the statement.

\begin{proposition}
\label{hoelderint}
Let $0<\widehat\alpha<1/2$, $p\in[2,\infty)$ and $\sigma$ be a strongly measurable 
adapted process, belonging to $L^{0}(\Omega; L^{p}(0,T;\mathcal{L}_{2}(H,Z)) )$. 
For $\frac{1}{p}<\widehat\alpha<\frac{1}{2}$ there exists a $\sigma$-independent positive 
constant $C_{T}$ which tends to $0$ as $T\searrow 0$, such that
		\begin{equation}
		||J(\sigma) ||_{L^{p}(\Omega; \cC^{\widehat\alpha-\frac{1}{p}} (0,T;Z))} 
		\leq C_{T} ||\sigma ||_{L^{p}(\Omega;L^{p}(0,T;\mathcal{L}_{2}(H,Z)) }.
		\end{equation}
\end{proposition}
	
\begin{remark}
	The assertions above remain valid for type-$2$ Banach spaces (e.g. $L^{p}$-spaces 
	for $p\geq 2$, Sobolev-spaces $W^{k,p}$ for $p\geq 2$), consult \cite{PronkVeraar} 
	and \cite{PronkVeraar1}. In this case one has to replace the Hilbert-Schmidt 
	operators by $\gamma$-radonifying ones. Given a separable Hilbert space $H$ and 
	a separable Banach space $Z$, we call an operator $\cR:H\to Z$ a $\gamma$-radonifying 
	operator if
	\begin{equation*}
	\mathbb{E}\left\|\sum\limits_{n=1}^{\infty}\gamma_{n}\cR e_{n} \right\|^{2}_{Z}<\infty,
	\end{equation*}
	where $(\gamma_{n})_{n\geq 1}$ is a sequence of independent standard Gaussian 
	random variables on $(\Omega,\mathcal{F},\mathbb{P})$ and $(e_{n})_{n\geq 1}$ is an 
	orthonormal basis in $H$. The space of $\gamma$-radonifying operators $\gamma(H,X)$ is 
	then endowed with the norm
	$$\left(\mathbb{E}\left\|\sum\limits_{n=1}^{\infty}\gamma_{n}\cR e_{n} 
	\right\|^{2}_{Z}\right)^{1/2},$$
	which does not depend on the choice of $(\gamma_{n})_{n\geq 1}$ and $(e_{n})_{n\geq 1}$.
	If $Z$ is isomorphic with a Hilbert space then $\gamma(H,Z)$ isometrically coincides 
	with $\mathcal{L}_{2}(H,Z)$. In summary, the computations in Section~\ref{qspde}  carry over to the Banach space-valued setting, although we present them
	here only in a Hilbert space setting.
\end{remark}	

%%%%%%%%%%%%%%%%%%%%%%%%%%%%%%%%%%%%%%%%%%%%%%%%%%%%%%%%%%%%%%%%%%%%%%%%%%%%%%%%%%%%%	
\section{Quasilinear SPDEs}
\label{qspde}

In this section we analyze stochastic quasilinear SPDEs using fixed-point arguments. 
In the deterministic case, this technique is known and can be found 
in \cite{Amann2}, \cite[Chapter 5]{Yagi1} or \cite{Yagi}. As already emphasized, 
since the linear part also depends on the solution itself, the corresponding 
parabolic evolution operators will no longer have the necessary measurability 
properties required to define the It\^{o}-integral, recall Proposition~\ref{meas}. 
Therefore, our ansatz is similar to the one used in \cite{PronkVeraar} to deal 
with parabolic SPDEs with time-dependent random generators. Combining this approach 
with the fixed-point arguments of \cite{Yagi1}, \cite{Yagi} and \cite{Amann2}, 
we are able to prove short-time existence for quasilinear SPDEs. In contrast to the non-autonomous random case we have to deal here with several technical difficulties, such as finding the appropriate function spaces for the fixed-point argument of \cite{Yagi5} and using the right localization techniques. Note that even 
in the deterministic case, quasilinear PDEs may not possess global-in-time solutions 
without further assumptions, see for instance \cite{Amann2}, \cite{Juengel1} 
and Section~\ref{examples}.\medskip

In the following, let $X,Y$ and $Z$ denote three separable Hilbert spaces 
such that $Z\hookrightarrow Y \hookrightarrow X$ and let $K$ stand for an arbitrary open ball in $Z$. More precisely $K:=\{V\in Z\mbox{ : } ||V||_{Z}< R\}$ for a deterministic fixed $R>0$.\\

The first step is to consider the quasilinear Cauchy problem
\begin{equation}
\begin{cases}\label{cp1}
~\txtd u(t)=\left[(Au(t))(u(t)) + f(t)\right]~\txtd t + \sigma(t) ~ \txtd W(t),~~t\in[0,T] \\
u(0)=u_{0}\in K \mbox{ a.s.}
\end{cases}
\end{equation}
 The results obtained for the inhomogenuous problem (\ref{cp1}) will be further 
extended to the nonlinear case, more precisely we will include nonlinearities of 
semilinear type. 

\begin{definition}
\label{definition:solution} \textbf{\em (Local solution)}
A local pathwise mild solution for (\ref{cp1}) is a pair $(u,\tau)$, where $\tau$ is a strictly positive stopping time and the stochastic process $\{u(t)\mbox{ : }t\geq 0\}$ is $(\mathcal{F}_{t})_{t\in[0,\tau)}$-adapted and satisfies almost surely for every $t\geq 0$
\begin{align}\label{lsg}
u(t)& = U^{u}(t,0)u_{0} + U^{u}(t,0)\int\limits_{0}^{t}\sigma(s)~\txtd W(s) 
+ \int\limits_{0}^{t} U^{u}(t,s)f(s) ~\txtd s\\
	&  - \int\limits_{0}^{t} U^{u}(t,s)A(u(s))\int\limits_{s}^{t}
	\sigma(r)~\txtd W(r) ~\txtd s.	\end{align}
\end{definition}

\begin{remark}
The concept \emph{pathwise mild solution} introduced in \cite{PronkVeraar} 
is justified by the integration by parts formula for stochastic convolutions
as motivated in \cite[Section 4.2]{PronkVeraar} as well as Appendix \ref{a}. 
In this way, one overcomes the difficulty that the It\^{o}-integral
\benn
\int\limits_{0}^{t}U^{u}(t,s,\omega)\sigma(s) ~\txtd W(s)  
\eenn
cannot be defined since the mapping $\omega\mapsto U^{u}(t,s,\omega)$ introduced 
in Section \ref{p} is according to Proposition \ref{meas} only 
$\mathcal{F}_{t}$-measurable and not $\mathcal{F}_{s}$-measurable. Furthermore, 
as shown in \cite[Theorem 3.4-3.5]{PronkVeraar} the convolution-type integrals 
above can be defined in a pointwise sense. This fact, together with certain pathwise 
regularity results, will be exploited in the construction of solutions below.
\end{remark}

\begin{definition}
A local \emph{pathwise mild solution} $\{u(t) \mbox{ : } t\in[0,\tau)\}$ 
for (\ref{cp1}) is unique if for any other local pathwise mild solution 
$\{\widetilde{u}(t) \mbox{ : } t\in[0,\widetilde{\tau})\}$ 
of (\ref{cp1}), the processes $u$ and $\widetilde{u}$ are equivalent 
on $[0,\tau\wedge\widetilde{\tau})$.
\end{definition}
	
\begin{definition}
\label{maxlocsol} \em(\textbf{Maximal and global solution})
We call $\{u(t) \mbox{ : } t\in[0,\tau)\}$  a \emph{maximal} local pathwise 
mild solution of (\ref{cp1}) if for any other local pathwise mild solution 
$\{\widetilde{u}(t): t\in[0,\widetilde{\tau})\}$ satisfying $\widetilde{\tau}(t)\geq 
\tau$ a.s. and $\widetilde{u}|_{[0,\tau)}$ is equivalent to $u$, one has 
$\widetilde{\tau}=\tau$ a.s. If $\{u(t)\mbox{ : } t\in[0,\tau)\}$ is a maximal 
local pathwise mild solution for (\ref{cp1}), then the stopping time $\tau$ is 
called its lifetime. If $\mathbb{P}(\tau=\infty)=1$ then $\{u(t) \mbox{ : } 
t\in[0,\tau)\}$ is a \emph{global} pathwise mild solution for (\ref{cp1}).
\end{definition}

\begin{remark}
We emphasize that by a maximal local pathwise mild solution of (\ref{multiplicative}) 
we understand a triple $(u,(\tau_{n})_{n\geq 1}, \tau_{\infty} )$ such that each 
pair $(u,\tau_{n})$ is a local pathwise mild solution, $(\tau_{n})_{n\geq 1}$ is an 
increasing sequence of stopping times such that $\tau_{\infty}:=\lim
\limits_{n\uparrow\infty} \tau_{n}$ a.s. and
$$\lim\limits_{t\uparrow\tau_{\infty}}\sup\limits_{s\in[0,t]} 
|| u(s)||_{Z}=\infty,\mbox{ a.s. on the set } \{\omega :\tau_{\infty}(\omega)<\infty\},$$
see also \cite[Proposition 3.11]{BrzezniakHausenblasRazafimandimby}. For a global 
solution it holds that $\tau_{\infty}=\infty$ a.s., which means that for every 
$T>0$ the quantity $$\sup\limits_{t\in[0,T]}|| u(t)||_{Z}$$ is almost surely finite 
on the set $\{\omega: \tau_{\infty}(\omega)=\infty\}$. Furthermore, note that if 
uniqueness of local solutions holds true, then the same remains valid for maximal 
local solutions~\cite[Section 3]{BrzezniakHausenblasRazafimandimby}.
\end{remark}

For more details regarding local and maximal local mild solutions for SPDEs 
consult \cite{ZhuBrzezniak}, \cite[Section 3]{BrzezniakHausenblasRazafimandimby}, 
\cite{Kim1} and the references specified therein.\medskip

In order to ensure the well-posedness of (\ref{cp1}) we further state suitable 
assumptions for the generators.

\begin{remark}
Note that since the operator $A$ depends on the solution itself it can degenerate 
at some point already in the deterministic setting. A simple example is the change 
from a forward to a backward heat equation. Of course, the same issue arises for 
general multi-component degenerate quasilinear parabolic problems~\cite{Amann3}.
The conditions we consider below will exclude this situation in a certain time-horizon. For more information on quasilinear degenerate SPDEs consult \cite{DebusscheHofmanovaVovelle}.
%%This implies that global well-posedness has to be checked  for every 
%%problem individually.
\end{remark}

\begin{definition}
Denote by $\mathcal{U}_{T}$ the set of all stochastic processes $(u(t))_{t\in[0,T]}$ satisfying the following conditions: 
\begin{itemize}
 \item $u(t)$ is $(\mathcal{F}_{t})_{t\in[0,T]}$-adapted;
 \item $u(t)$ has a.s.~continuous trajectories in $X$;
 \item $u(t)$ has a.s.~bounded trajectories in $Z$ and satisfies $||u(t) -u_{0}||_{Z} < R$~a.s.~for 
$t\in[0,T]$.
\end{itemize}
\end{definition}
Recall that $R>0$ (deterministic) was introduced at the beginning of this section. In the following sequel, $\mathcal{B}([0,T];Z)$ stands for the Banach space of all bounded $Z$-valued functions on $[0,T]$.\\

Similar to \cite{Yagi1} and \cite{Yagi5} one can make certain structural assumptions which ensure that $A(u)$ exists for $u\in \mathcal{U}_{T}$ and generates an evolution system $U^{u}$. However, since all these assertions will depend on $u$ we localize them as follows.\\

We introduce the sequence of stopping times $(\tau_{n})_{n\geq 1}$ as 
\be
\label{eq:stoptime}
\tau_{n}:=\inf \{t\geq 0: ||u(t)||_{Z} \geq n\}
\ee
and impose several main assumptions (A1')-(A3') that ensure the 
existence of a parabolic evolution operator $U^{u}(\cdot,\cdot)$ together with a 
local Lipschitz continuity of the generators $A(\cdot{})$: 

\begin{itemize}
	\item [(A1')] For $u\in \mathcal{U}_{T \wedge\tau_{n}}$, $A(u)$ is a sectorial 
		operator of angle $0<\varphi<\frac{\pi}{2}$, namely
		$$\sigma(A(u))\subset \Sigma_{\varphi}:=\{ \lambda\in\mathbb{C}: 
		|\mbox{arg} \lambda|<\varphi \}, \mbox{ for } u\in \mathcal{U}_{T\wedge\tau_{n}} .  $$
	\item [(A2')]  For $u\in \mathcal{U}_{T\wedge\tau_{n}}$ the resolvent operator 
		$(\lambda \mbox{Id} -A(u))^{-1}$ satisfies the Hille-Yosida estimate, i.e., there exists 
		$M\geq 1$ such that
		$$||(\lambda \Id -A(u))^{-1}||_{\mathcal{L}(X)} \leq \frac{M} {|\lambda|+1},\mbox{ 
		for } \lambda\notin\Sigma_{\varphi}, \mbox{ for } u\in \mathcal{U}_{T\wedge\tau_{n}}.$$
	\item [(A3')] Let $0<\nu\leq 1$ be fixed. Then
  \begin{equation}\label{nu}
	||A^{\nu}(u) (A(u)^{-1}- A(v)^{-1}) ||_{\mathcal{L}(X)}\leq 
	n || u-v||_{Y}, \mbox{ for } u,v\in \mathcal{U}_{T\wedge\tau_{n}}.
	\end{equation}
 Furthermore, recall that $Y$ is a third Hilbert space such that $Z\hookrightarrow Y \hookrightarrow X$.%% Moreover we have
%% \begin{align}\label{b:domain}
 %%||A(u) ||_{\mathcal{L}( D(A(u)),X)} \leq n, \mbox{ for } u \in \cU_{T\wedge\tau_{n}}.
 %%\end{align}

\end{itemize}

Assumptions (A1') and (A2') imply according to Theorem \ref{po} that $A(u)$ generates 
an evolution system $U^{u}$ for all $u\in \mathcal{U}_{T\wedge\tau_{n}}$.\\

Moreover, regarding (\ref{est:u}) we infer that
\begin{align*}
||U^{u}(t,s)||_{\mathcal{L}(X)} \leq \widetilde{C}(\delta) e ^{n^{1/\delta}(t-s)}, ~~\mbox{for } u\in \mathcal{U}_{T\wedge\tau_{n}},
\end{align*}
 so
 \begin{align}\label{est:u:stopped}
 \sup\limits_{0\leq s \leq t \leq T\wedge\tau_{n} } ||U^{u}(t,s)||_{\mathcal{L}(X)} \leq C (\delta) e^{Tn^{1/\delta}}, ~~\mbox{for } u\in \mathcal{U}_{T\wedge\tau_{n}}.
 \end{align}
\begin{itemize}
	\item [(A4')] We require that there exist $\alpha,\beta$ with 
	$0\leq \alpha<\beta<\nu\leq 1$, $\beta+\nu>1+\alpha$ and $\beta\geq 1/2$, such that 
	$D(A(u)^{\alpha})=:Y$ and $D(A(u)^{\beta})=:Z$; see \cite[Chapter 5, 
	Section 1.1]{Yagi1}. %% Furthermore, we also impose $\alpha<\gamma$, where $\gamma$ stands for the 
%%	H\"older-regularity of the stochastic integral, recall Proposition \ref{reg}.
\end{itemize}
	
	\begin{remark}
		\begin{itemize}
			\item 	We emphasize that the domains $D(A(u))$ are in general allowed to depend on $u$
			as discussed in~\cite{Amann2}, \cite[Chapter 5]{Yagi1} and \cite[Section 3]{Yagi5}. We assume here for 
			the case of brevity only constant domains (i.e. $\nu=1$ in (A3')), but the techniques applied in this framework 
			can be extended to non-constant domains assuming for instance $D(A(v)) \subset D (A(u)^{\nu})$ for $u,v\in\cU_{T\wedge\tau_{n}}$ and $D(A(u)^{\beta})
			\hookrightarrow Z$, $D(A(u)^{\alpha})\hookrightarrow Y$. 
			\item Note that one can take $\alpha=0$ which means that $Y=X$, see \cite[Chapter 5]{Yagi}.
			\item Instead of (\ref{nu}) one can impose a global Lipschitz continuity of the generators, see \cite{Hornung1} and thereafter introduce a cut-off function.
		\end{itemize}
	
	\end{remark}
	
The assumption (A4') is meaningful from the point of 
view of the applications we want to consider as justified in Section \ref{app}.
Since we are in the parabolic setting, we specify that we can identify the 
domains of the fractional powers of these generators with Sobolev spaces, namely
\benn
X_{\widehat\mu} :=D(A(u)^{\widehat\mu}) =H^{2\widehat\mu}(G)\quad  \text{or}\quad 
D(A(u)^{\widehat\mu})=H^{2\widehat\mu}_{N}(G)\quad \text{or}\quad D(A(u)^{\widehat\mu})= H^{2\widehat\mu}_{D}(G),
\eenn
depending on the range of $\widehat\mu\geq 0$ and on the boundary conditions of $A$, see 
Section \ref{app}. Here $D$ and $N$ stand for Dirichlet respectively Neumann 
boundary conditions and $G\subset \mathbb{R}^{n}$ is an open bounded $\cC^{2}$-domain; 
Section \ref{app} provides concrete examples for this setting. In the following we 
let $||\cdot{}||_{\widehat\mu}$ denote the norm in $X_{\widehat\mu}$.\medskip
 
 We introduce a time-horizon $\widetilde{T}$ such that $0<\widetilde{T}\leq T$ which will be chosen small enough 
 as required in the fixed-point argument presented below and specify the set of processes we consider. Note that the following definition is meaningful since all terms involved in our computation exist pathwise.
 \begin{definition}\label{def:p}
 	  We define for $1-\nu<\delta<\beta-\alpha \wedge\gamma$ and an arbitrary positive 
 	  deterministic constant $k$ the set $\mathcal{K}$ of all $(\mathcal{F}_{t})_{t\in[0,\widetilde{T}\wedge\tau_{n}]}$-adapted stochastic processes $u:[0,\widetilde{T}\wedge\tau_{n}]\times\Omega\to Y$ such that
 	  \begin{itemize}
 	  	\item [1)] $u(0)=u_{0}$ a.s.
 	  	\item [2)] $u\in\mathcal{B}([0,\widetilde{T}\wedge\tau_{n}];Z)$ with
 	  	$
 	  	\sup\limits_{0\leq t \leq \widetilde{T}\wedge\tau_{n}} ||u(t) - u_{0}||_{Z}\leq r \mbox{ a.s.}
 	  	$
 	  	\item [3)] $u\in \cC^{\delta}([0,\widetilde{T}\wedge\tau_{n}];Y)$ with 
 	  	$
 	  	\sup\limits_{0\leq t \leq\widetilde{T}\wedge\tau_{n}} \frac{||u(t)-u(s)||_{Y}}{(t-s)^{\delta}} \leq k \mbox{ a.s.}
 	  	$
 	  	
 	  \end{itemize}

 \end{definition}

%%\begin{align*}
%%\mathcal{K}:=\Big\{u\in \cB([0,\widetilde{T}\wedge\tau_{n}];Z) 
%%\cap \cC^{\delta}([0,\widetilde{T}\wedge\tau_{n}];Y)\mbox{ : } 
%%u(0)=u_{0}, \sup\limits_{0\leq t \leq \widetilde{T}\wedge\tau_{n}} 
%%||u(t)-u_{0}||_{Z} \leq R \mbox{ a.s.}\\
%%\mbox{and }\sup\limits_{0\leq s< t \leq \widetilde{T}\wedge\tau_{n}} 
%%\frac{||u(t) -u(s)||_{Y}}{(t-s)^{\delta}}\leq k \mbox{ a.s.} \Big \}.
%%\end{align*}

Here we let $0<r<R$, where the deterministic constant $R$ introduced above describes the radius of an arbitrary open ball in $Z$ and $\gamma$ denotes the H\"older exponent of $J(\sigma)$, recall Proposition \ref{hoelderint}. 

\begin{remark}
	\begin{itemize}
		\item For computational simplicity, it is enough to consider $\cC([0,\widetilde{T}\wedge\tau_{n}];Y)$ instead of $\cC^{\delta}([0,\widetilde{T}\wedge\tau_{n}];Y)$, i.e.~one shows that the trajectories of $u$ are a.s.~bounded in $Z$ and a.s.~continuous in $Y$. For optimal space-time regularity results we work with $\mathcal{K}$ as defined above. 
		\item Note that $\mathcal{K}$ is a closed subset of $\cC([0,\widetilde{T}\wedge\tau_{n}];Y)$.
	\end{itemize}
	
\end{remark}

The first key result we prove is: 
\begin{theorem}
\label{sol} 
Let $f\in \cC^{\delta}([0,T];X)$, $\sigma\in L^{0}(\Omega; \cC([0,T];\mathcal{L}_{2}(H,X_{2\beta})))$ and (A1')-(A4') be satisfied.
 Then the stochastic evolution equation (\ref{cp1}) has 
a unique local pathwise mild solution $u\in L^{0}(\Omega; 
\cB([0,\widetilde{T}\wedge\tau_{n}];Z))\cap L^{0}
(\Omega;\cC^{\delta}([0,\widetilde{T}\wedge\tau_{n}];Y))$.
\end{theorem}

\begin{remark}
	\begin{itemize}
	\item [1)] The assumption $\sigma(\cdot{})\in \mathcal{L}_{2}(H,X_{2\beta})$ will be 
	required in order to prove the contraction property and is justified in 
	Lemma~\ref{stochconv}. For suitable estimates of the generalized stochastic 
	convolution given in (\ref{definition:solution}), one only needs 
	$\sigma(\cdot{})\in \mathcal{L}_{2}(H,Z)$ as indicated in the following computations. 
	For similar regularity conditions and further applications we
	refer the reader to \cite[Section 2]{ChuesovScheutzow}.
	\item [2)] Note that the representation formula (\ref{pms}) holds true under this additional space-regularity assumption on $\sigma$.
	\end{itemize}
\end{remark}

We proceed to the proof of Theorem \ref{sol}. To this aim, let $v\in \mathcal{K}$ a.s.
and let $A_{v}(t)$ denote the family of sectorial operators $A_{v}(t):=A(v(t,\omega))$. 
For notational simplicity, the $\omega$-dependence will be dropped
but it has to be kept in mind. We first consider the evolution equation

\begin{equation}
\label{v}
\begin{cases}
~\txtd u(t)=\left[A_{v}(t)u(t) + f(t) \right] ~\txtd t + \sigma(t) ~\txtd W(t), 
~ t\in[0,\widetilde{T}]\\
u(0)=u_{0}\in K \mbox{ a.s.}
\end{cases}
\end{equation}
Note that (\ref{v}) represents a linear parabolic stochastic Cauchy problem, 
with time-dependent, random drift. 
Using \cite[Theorem 5.3]{PronkVeraar} we infer that (\ref{v}) has a pathwise 
mild solution given by
\begin{align}
u(t) & = U^{v}(t,0)u_{0} + U^{v}(t,0)\int\limits_{0}^{t}\sigma(s)~\txtd W(s) 
+ \int\limits_{0}^{t}U^{v}(t,s)f(s) ~\txtd s\nonumber \\
& - \int\limits_{0}^{t}U^{v}(t,s)A_{v}(s) \int\limits_{s}^{t}
\sigma(r) ~\txtd W(r)~\txtd s~ \mbox{   a.s.},\label{eq:magain}
\end{align}
where $U^{v}(t,s)$ is the random parabolic evolution operator generated 
by $A_{v}$. 

\begin{remark}
Note that using \cite[Theorems 3.4-3.5]{PronkVeraar}, the previous convolutions 
can be defined in a pointwise sense and are pathwise continuous, which means that 
the solution formula~\eqref{eq:magain} holds for almost all $\omega\in\Omega$ and 
for all $t\in[0,\widetilde{T}]$; we also refer to the proof of 
Lemma~\ref{sa} below.	
\end{remark}

%%\begin{remark}
%%\label{omega}
%%We emphasize once more the fact that the Lipschitz constant in (\ref{nu}) is 
%%bounded in $\Omega$, see \cite[Section 5.2]{PronkVeraar}. This is crucial and 
%%implies that the universal constant $C>0$ in the following computation does not 
%%depend on $\omega$. 
%%\end{remark}

We define the mapping 
$$\Phi(v):=u , \mbox{ for } v\in \mathcal{K}$$ 
and are going to prove that $\Phi$ maps $\mathcal{K}$ into itself and that it 
is a contraction with respect to the norm in $L^{2}(\Omega; \cC([0,\widetilde{T}\wedge\tau_{n}];Y))$ if one chooses $\widetilde{T}$ small enough. We split the proof 
into several steps.\\

Throughout this section we frequently use 
the estimates for analytic $C_{0}$-semigroups and parabolic evolution operators as stated 
in Theorem \ref{estan}. Namely, according to (\ref{an}), (\ref{an2}) there exists a constant $C:\Omega\to\mathbb{R}_{+}$ such that for $0\leq s<t \leq \widetilde{T}\wedge\tau_{n}$ and $\widetilde{\alpha},\widetilde{\beta}\in(0,1]$ we have
\begin{equation}
||A^{\widetilde\beta}_{v}(t) U^{v}(t,s) ||_{\mathcal{L}(X)} \leq C (t-s)^{-\widetilde\beta},\quad  
 ||A^{\widetilde\beta}_{v}(t) U^{v}(t,s)A^{-\widetilde\alpha}_{v}(s) ||_{\mathcal{L}(X)}
\leq C (t-s)^{\widetilde\alpha-\widetilde\beta}.
\end{equation}

 Furthermore, one can estimate the difference between $U^{v}(t,s)$ and $e^{-(t-s)A_{v}(t)}$ or $e^{-(t-s)A_{v}(s)}$ for $0\leq s\leq t \leq \widetilde{T}\wedge\tau_{n}$ as follows from \cite[Section 3]{Yagi5}:
\begin{align}
||A^{\widetilde{\theta}}_{v}(t)U^{v}(t,s) A ^{-\widetilde{\theta}}_{v}(s) - e ^{(t-s) A_{v}(s)} ||_{\cL(X)} \leq C (t-s) ^{\delta +\nu -1}, ~~ 0\leq \widetilde{\theta}\leq 1 \label{id1}\\
 || A^{\widetilde{\theta}}_{v}(t) (U^{v}(t,s) - e^{-(t-s)A_{v}(t)}) A^{-\widetilde{\rho}}_{v}(s) ||_{\cL(X)} \leq C (t-s) ^{\widetilde{\rho} -\widetilde{\theta}+\delta+\nu-1 }, ~~ 0 \leq \widetilde{\theta}, \widetilde{\rho}\leq 1 \label{id2}.
\end{align}
Recall that $\delta$ stands for the H\"older exponent in the Acquistapace-Terreni condition (\ref{at}). 
\begin{remark}
	Note that $C$ depends on the stopping times, more precisely, according to (\ref{est:u:stopped}) one has $C\leq \widetilde{C}(\delta) e^{\widetilde{T}n ^{1/\delta}}$. For notational simplicity we drop this dependence, but it should be kept in mind. 
\end{remark}

Regarding all these we proceed towards our fixed-point argument.

\begin{lemma}
\label{sa}
If $\widetilde{T}$ is sufficiently small, then $\Phi$ maps 
$\mathcal{K}$ into itself.
\end{lemma}

\begin{proof}
The regularity results stated in Proposition \ref{reg} are the key ingredients, 
which are required in order to estimate pathwise the terms containing stochastic 
integrals. Keeping in mind that $Z=D(A^{\beta}_{v}(t))$ for 
$t\in[0,\widetilde{T}\wedge\tau_{n}]$ and $v\in\mathcal{K}$ a.s., the first of these 
estimates entails 
\begin{align*}
	&	\sup\limits_{0\leq t \leq \widetilde{T} \wedge\tau_{n} } || U^{v}(t,0)
	\int\limits_{0}^{t}\sigma(s) ~\txtd W(s)||_{Z}=\sup\limits_{0\leq t \leq 
	\widetilde{T} \wedge\tau_{n} } || A^{\beta}_{v}(t)U^{v}(t,0) 
	\int\limits_{0}^{t}\sigma(s) ~\txtd W(s)||_{X}\\
	&= \sup\limits_{0\leq t \leq \widetilde{T} \wedge\tau_{n} } 
	|| A^{\beta}_{v}(t) U^{v}(t,0) A^{-\beta}_{v}(0) A^{\beta}_{v}(0) 
	\int\limits_{0}^{t}\sigma(s) ~\txtd W(s)||_{X} \\
	& \leq C \sup\limits_{0\leq t \leq \widetilde{T}\wedge\tau_{n}}
	||A^{\beta}_{v}(t) U^{v}(t,0) A^{-\beta}_{v}(0) ||_{\mathcal{L}(X)} 
	||A^{\beta}_{v}(0) \int\limits_{0}^{t} \sigma(s) ~\txtd W(s)  ||_{X}\\
	&  \leq C \sup\limits_{0\leq t \leq\widetilde{T}\wedge\tau_{n}} 
	||\int\limits_{0}^{t}\sigma(s) ~\txtd W(s) ||_{Z}
	\leq C ||J(\sigma) ||_{\cC([0,\widetilde{T}\wedge\tau_{n}];Z)}.
\end{align*}
Furthermore, due to Proposition \ref{reg}, $J(\sigma)\in \cC^{\gamma}
([0,\widetilde{T}\wedge\tau_{n}];Z)$ a.s. for $\gamma<1/2$. We obtain
	\begin{align*}
	& \sup\limits_{0\leq t \leq \widetilde{T}\wedge\tau_{n}} 
	||\int\limits_{0}^{t} U^{v}(t,s) A_{v}(s) \int\limits_{s}^{t}\sigma(r) ~\txtd W(r) 
	~\txtd s ||_{Z}
	%= \sup\limits_{0\leq t \leq \widetilde{T}\wedge\tau_{n}} 
	%||A^{\beta}_{v}(t)\int\limits_{0}^{t} U^{v}(t,s) A_{v}(s) \int\limits_{s}^{t}
	%\sigma(r) ~\txtd W(r) ~\txtd s ||_{X}
	\\
	&= \sup\limits_{0\leq t \leq \widetilde{T}\wedge\tau_{n}} 
	||\int\limits_{0}^{t} A^{\beta}_{v}(t) U^{v}(t,s) A_{v}(s) A^{-\beta}_{v}(s) 
	A^{\beta}_{v}(s) \int\limits_{s}^{t} \sigma(r) ~\txtd W(r) ~\txtd s||_{X} \\
	& \leq C  \sup\limits_{0\leq t \leq \widetilde{T} \wedge\tau_{n} }
	\int\limits_{0}^{t} || A^{\beta}_{v}(t)U^{v}(t,s)A_{v}(s)A^{-\beta}_{v}(s) 
	||_{\mathcal{L}(X)} ||A^{\beta}_{v}(s)\int\limits_{s}^{t} \sigma(r) ~\txtd W(r) 
	||_{X} ~\txtd s  \\
	&\leq C \sup\limits_{0\leq t \leq \widetilde{T}\wedge\tau_{n} } 
	\int\limits_{0}^{t} (t-s)^{-1} || \int\limits_{s}^{t} \sigma(r) ~\txtd W(r) 
	||_{Z} ~\txtd s\\
	& \leq C \sup\limits_{0\leq t \leq \widetilde{T}\wedge\tau_{n}} \int\limits_{0}^{t}
	(t-s)^{ -1} (t-s)^{\gamma} ||J(\sigma)||_{\cC^{\gamma}([0,\widetilde{T}
	\wedge\tau_{n}];Z)}  ~\txtd s
	\leq  C \widetilde{T}^{\gamma} ||J(\sigma) ||_{\cC^{\gamma}([0,
	\widetilde{T}\wedge\tau_{n}];Z) }.
	\end{align*}
We can also estimate the term involving the initial condition 
\begin{align*}
\sup\limits_{0\leq t \leq \widetilde{T}\wedge\tau_{n}} || U^{v}(t,0)u_{0}||_{Z}
&={ \color{black}{ 
%\sup\limits_{0 \leq t \leq \widetilde{T} \wedge\tau_{n}} 
%||A^{\beta}_{v}(t)U^{v}(t,0)u_{0} ||_{X}
\sup\limits_{0\leq t \leq \widetilde{T}
\wedge\tau_{n}} ||A^{\beta}_{v}(t) U^{v}(t,0) A^{-\beta}_{v}(0) A^{\beta}_{v}(0)
u_{0} ||_{X} }} \\
	&{\color{black}{\leq C \sup\limits_{0\leq t \leq \widetilde{T}\wedge\tau_{n} } 
	|| A^{\beta}_{v}(t) U^{v}(t,0) A^{-\beta}_{v}(0) ||_{\mathcal{L}(X)}
	||A^{\beta}_{v}(0)u_{0} ||_{X}}}\leq C ||u_{0} ||_{Z}.
	\end{align*}
We immediately obtain that 
\begin{align*}
	& \sup\limits_{0\leq t \leq \widetilde{T}\wedge\tau_{n}} 
	||\int\limits_{0}^{t} U^{v}(t,s)f(s) ~\txtd s ||_{Z}
	 \leq C \widetilde{T}^{1-\beta}||f ||_{\cC^{\delta}([0,\widetilde{T}
	\wedge\tau_{n}];X)}.
\end{align*}
Consequently, we may conclude based upon the previous estimates that
\begin{align*}
\sup\limits_{0\leq t\leq\widetilde{T}\wedge\tau_{n}} 
||u(t)||_{Z}&\leq C\left( ||u_{0} ||_{Z} +  \widetilde{T}^{\gamma}
||J(\sigma) ||_{\cC^{\gamma}([0,\widetilde{T}\wedge\tau_{n}];Z)} 
+ \widetilde{T} ^{1-\beta} || f||_{\cC^{\delta}([0,
\widetilde{T}\wedge\tau_{n}];X)}\right).
\end{align*}
Furthermore, one can derive regularity results for $u$ in 
appropriate function spaces. This can be shown using arguments 
from \cite[Proposition 5.1]{Yagi1} (deterministic estimates) and 
\cite[Theorem 4.4]{PronkVeraar} 
(pathwise Sobolev/H\"older regularity of the generalized stochastic convolution). 
For the convenience of the reader, we shortly indicate the main computation which justifies the H\"older-continuity of $u$ in $Y$.
 Let $0\leq s<t \leq 
\widetilde{T}\wedge\tau_{n}.$ Then building the difference of the solution at these two time points yields
\begin{align*}
u(t) - u(s) &= (U^{v} (t,s) - \mbox{Id} ) u(s) + U^{v}(t,0) \int\limits_{s}^{t} \sigma(r) ~\txtd W(r) + \int\limits_{s}^{t} U^{v}(t,\tau) f(\tau) ~\txtd \tau\\
& + \int\limits_{s}^{t} U ^{v}(t,\tau) A_{v} (\tau) \int\limits_{\tau}^{t} \sigma(r) ~\txtd W(r) ~\txtd \tau + \int\limits_{0}^{s} U^{v}(t,\tau) A_{v}(\tau) \int\limits_{s}^{t} \sigma(r)~\txtd W(r) ~\txtd \tau.
\end{align*} 

We write the first term as
\begin{align}\label{diffh}
(U^{v}(t,s)- \mbox{Id} ) u(s) &= \Big[ (U^{v}(t,s) - e ^{-(t-s)A_{v}(t)} ) A^{-\beta}_{v}(s) + (e^{-(t-s)A_{v}(t) } -\mbox{Id})(A^{-\beta}_{v}(s) - A^{-\beta}_{v}(t) ) \nonumber\\
&+ (e^{-(t-s)A_{v}(t)} -\mbox{Id}) A^{-\beta}_{v}(t)   \Big] A^{\beta}_{v}(s) u(s).
\end{align}

Recalling that $||\cdot||_{Y}=||A^{\alpha}_{v}(t) \cdot ||_{X}$ we start estimating all the terms above in the appropriate norm. Using (\ref{id2}) we have for the first term
\begin{align*}
||A^{\alpha}_{v}(t) (U^{v}(t,s) - e ^{-(t-s)A_{v}(t)} ) A^{-\beta}_{v}(s) ||_{\cL(X)} \leq C (t-s) ^{\beta-\alpha+\delta +\nu-1}.
\end{align*}
Recall that for $\delta$ introduced in (\ref{at}) we imposed that $1-\nu<\delta<\beta-\alpha\wedge\gamma$. \\

Similarly,
\begin{align*}
&||A^{\alpha}_{v}(t) (e^{-(t-s)A_{v}(t) } -\mbox{Id})(A^{-\beta}_{v}(s) - A^{-\beta}_{v}(t) )||_{\cL(X)}\\
& \leq ||(e^{-(t-s)A_{v}(t) } -\mbox{Id}) A^{\alpha-\alpha'}_{v}(t) ||_{\cL(X)} ||A^{\alpha'}_{v}(t) (A^{-\beta}_{v}(s) -A^{-\beta}_{v}(t )) ||_{\cL(X)} \leq C (t-s)^{\alpha'-\alpha+\delta}
\end{align*}
and
\begin{align*}
||A^{\alpha}_{v}(t)  (e^{-(t-s)A_{v}(t)} -\mbox{Id}) A^{-\beta}_{v}(t)||_{\cL(X)} \leq C (t-s)^{\beta-\alpha}.
\end{align*}
Here $\alpha'$ is an exponent satisfying $\alpha<\alpha'<\beta+\nu-1$. %as implied by (A3) .\\

Obviously,
\begin{align*}
\left\| A^{\alpha}_{v}(t) \int\limits_{s}^{t} U^{v}(t,\tau) f(\tau) ~\txtd \tau \right\|_{X}\leq C (t-s)^{1-\alpha} ||f||_{\cC^{\delta}([0,\widetilde{T}\wedge\tau_{n}];X)}.
\end{align*}

Now, we turn to the terms of (\ref{diffh}) that contain stochastic integrals and verify their regularity.\\

First of all
\begin{align*}
\left\| A^{\alpha}_{v}(t) U^{v}(t,0) A^{-\beta}_{v}(0) A^{\beta}_{v}(0) \int\limits_{s}^{t} \sigma(r) ~\txtd W(r) \right\|_{X} & \leq C t ^{\beta-\alpha} ||J(\sigma)(t) - J(\sigma)(s)||_{Z}\\& \leq C t^{\beta-\alpha} (t-s)^{\gamma} ||J(\sigma)||_{\cC^{\gamma}([0,\widetilde{T}\wedge\tau_{n}];Z)}.
\end{align*}
Furthermore
\begin{align*}
&\left\| A^{\alpha}_{v}(t)\int\limits_{s}^{t} U ^{v}(t,\tau) A_{v} (\tau) \int\limits_{\tau}^{t} \sigma(r) ~\txtd W(r) ~\txtd \tau \right\|_{X}\\
& = \left\| A^{\alpha}_{v}(t)\int\limits_{s}^{t} U ^{v}(t,\tau) A_{v} (\tau) A^{-\beta}_{v}(\tau) A^{\beta}_{v}(\tau) \int\limits_{\tau}^{t} \sigma(r) ~\txtd W(r) ~\txtd \tau \right\|_{X}  \\
& \leq C \int\limits_{s}^{t} (t-\tau) ^{\beta-1-\alpha} ||J(\sigma)(t) - J(\sigma) (\tau)||_{Z} ~\txtd \tau\\
& \leq C  (t-s)^{\beta-\alpha+\gamma} ||J(\sigma)||_{\cC^{\gamma}([0,\widetilde{T}\wedge\tau_{n}];Z)}.
\end{align*}
Finally

\begin{align*}
&\left\| A^{\alpha}_{v}(t)\int\limits_{0}^{s} U^{v}(t,\tau) A_{v}(\tau) \int\limits_{s}^{t} \sigma(r)~\txtd W(r) ~\txtd \tau \right\|_{X}\\
& = \left\| A^{\alpha}_{v}(t)\int\limits_{0}^{s} U^{v}(t,\tau) A_{v}(\tau) A^{-\beta}_{v}(\tau)  A^{\beta}_{v}(\tau) \int\limits_{s}^{t} \sigma(r)~\txtd W(r) ~\txtd \tau \right\|_{X}\\
& \leq (t-s)^{\beta-\alpha+\gamma} ||J(\sigma)||_{\cC^{\gamma}([0,\widetilde{T}\wedge\tau_{n}];Z)}. 
\end{align*}
%%\vspace*{0.3 cm}
In conclusion, 
choosing $\widetilde{T}$ sufficiently small, we have that 
$||u(t)-u_{0}||_{Z}\leq R$ and $||u ||_{\cC^{\delta}([0,\widetilde{T}\wedge 
\tau_{n}];Y)}\leq k$ a.s. This means that $\Phi$ maps $\mathcal{K}$ into itself 
as claimed. \qed
\end{proof}\\

We also provide mean-square estimates for $u$. The arguments employed in this computation will be required later on when we prove the contraction property of $\Phi$ with respect to the norm in $L^{2}(\Omega; \cC ([0,\widetilde{T}\wedge\tau_{n}];Y))$.
\begin{lemma}
\label{meansquare}
We have
\begin{align*}
	 \mathbb{E}\left[\sup\limits_{0\leq t \leq \widetilde{T}\wedge\tau_{n}} 
	||u(t) -u_{0}||_{Z}^{2} \right]\leq C  R^{2} + C\left( \widetilde{T}^{2(
	1-\beta)}||f||^{2}_{\cC^{\delta}([0,\widetilde{T}\wedge\tau_{n}];X)} + 
	\widetilde{T}||\sigma||^{2}_{\cC([0,\widetilde{T}\wedge\tau_{n}];
	\mathcal{L}_{2}(H,Z))} \right)
	\end{align*}
	as well as
	$$\mathbb{E} \left[||u||^{2}_{\cC^{\delta}([0,\widetilde{T}\wedge
	\tau_{n}];Y) }\right]\leq k^{2}. $$
\end{lemma}

\begin{proof}
Let $0<t\leq \widetilde{T}\wedge\tau_{n}$. We start to prove the first assertion
considering each term of the solution separately.
By similar deliberations combined with (\ref{id1}) and (\ref{id2}) we have 
\begin{align*}
 || (U^{v}(t,0)-\mbox{Id})u_{0}||_{Z} & = ||A^{\beta}_{v}(t) (U^{v}(t,0) 
-\mbox{Id} )u_{0} ||_{X}   
%{\color{blue}{|| A^{\beta}_{v}(t) (U^{v}(t,0) 
%-\mbox{Id}) A^{-\beta}_{v}(0) A^{\beta}_{v}(0)u_{0}||_{X}}}
\\
 & \leq C || A^{\beta}_{v}(t) (U^{v}(t,0) -\mbox{Id}) 
A^{-\beta}_{v}(0)||_{\mathcal{L}(X)} ||A^{\beta}_{v}(0)u_{0}||_{X} 
% \\ & \leq C ||u_{0} ||_{\beta},
%\end{align*}
%which means that 
%||(U^{v}(t,0)-\mbox{Id})u_{0} ||_{Z}
\leq C R, 
\end{align*}
since $u_{0}\in K$ a.s. Consequently, we get
\begin{align*}
\mathbb{E}	\left[\sup\limits_{0\leq t \leq\widetilde{T}\wedge\tau_{n}} 
||U^{v}(t,0)u_{0}-u_{0}||_{Z}^{2} \right]\leq C R^{2}.
\end{align*}
As argued above
\begin{align*}
 \mathbb{E} \left[\sup\limits_{0\leq t \leq \widetilde{T}\wedge\tau_{n}} 
|| \int\limits_{0}^{t} U^{v}(t,s) f(s) ~\txtd s ||^{2}_{Z}\right] 
\leq C \widetilde{T}^{2(1-\beta)} || f||^{2}_{\cC([0,\widetilde{T}
\wedge\tau_{n}];X)}.
\end{align*}
We now estimate the terms containing the stochastic integrals 
in $L^{2}(\Omega,Z)$ following the strategy in \cite[Lemma 5.2]{PronkVeraar}. The 
first stochastic integral yields
\begin{align*}
	|| U^{v}(t,0)\int\limits_{0}^{t}\sigma(s) ~\txtd W(s)|| _{L^{2}(\Omega,Z)}& 
	= ||A^{\beta}_{v}(t)  U^{v}(t,0) \int\limits_{0}^{t}\sigma(s) ~\txtd 
	W(s)||_{L^{2}(\Omega,X)}\\
%&{ \color{blue}{= ||A^{\beta}_{v}(t) U^{v}(t,0) A^{-\beta}_{v}(0) 
%A^{\beta}_{v}(0) \int\limits_{0}^{t} \sigma(s) ~\txtd W(s) 
%||_{L^{2}(\Omega;X)} }}\\
	& \leq C ||A^{\beta}_{v}(t)  U^{v}(t,0) |_{\mathcal{L}(Z,X)} 
	||\int\limits_{0}^{t}\sigma(s) ~\txtd W(s) ||_{L^{2}(\Omega,Z)} \\
& \leq C\sqrt{t} ||\sigma ||_{\cC([0,t];\mathcal{L}_{2}(H,Z))}.
\end{align*}
Applying the Burkholder-Davis-Gundy inequality gives
\begin{align*}
\mathbb{E}\left[ \sup\limits_{0\leq t \leq \widetilde{T}\wedge\tau_{n} } 
|| U^{v}(t,0)\int\limits_{0}^{t}\sigma(s) ~\txtd W(s)  
||^{2}_{Z}\right] \leq C \widetilde{T} ||\sigma ||^{2}_{\cC([0,\widetilde{T}
\wedge\tau_{n}];\mathcal{L}_{2}(H,Z))}.
\end{align*}
The next step is to analyze the generalized stochastic convolution. This yields
	\begin{align*}
& \left\| \int\limits_{0}^{t} U^{v}(t,s) A_{v}(s) \int\limits_{s}^{t} 
\sigma(r) ~\txtd W(r) ~\txtd s \right\|_{L^{2}(\Omega,Z)} = \left\| 
\int\limits_{0}^{t}A^{\beta}_{v}(t) U^{v}(t,s) A_{v}(s) 
\int\limits_{s}^{t} \sigma(r) ~\txtd W(r) ~\txtd s \right\|_{L^{2}(\Omega,X)}\\
	& \leq C \int\limits_{0}^{t} ||A^{\beta}_{v}(t) U^{v}(t,s)A_{v}(s)  
	||_{\mathcal{L}(Z,X)} ||\int\limits_{s}^{t} \sigma(r) ~\txtd W(r) 
	||_{L^{2}(\Omega,Z)} ~\txtd s \\
	& \leq C \int\limits_{0}^{t} (t-s)^{-1}(t-s)^{1/2} ||\sigma
	||_{\cC([0,t];\mathcal{L}_{2}(H,Z))} ~\txtd s
	\leq C \sqrt{t} ||\sigma ||_{\cC([0,t];\mathcal{L}_{2}(H,Z))}.
\end{align*}
Furthermore, one easily obtains that
%	\begin{align*}
%	\sup\limits_{0\leq t \leq \widetilde{T}\wedge\tau_{n} }
%	 \left\| \int\limits_{0}^{t} U^{v}(t,s)A_{v}(s)\int
%	\limits_{s}^{t} \sigma(r) ~\txtd W(r) ~\txtd s \right\|_{Z} 
%	\leq \int\limits_{0}^{\widetilde{T}\wedge\tau_{n}  } 
%	|| U^{v}(t,s)A_{v}(s) \int\limits_{s}^{t} \sigma(r)~\txtd W(r)||_{Z} ~\txtd s,
%	\end{align*}
%which gives us 
\begin{align*}
&\mathbb{E}\left[\sup\limits_{0\leq t \leq \widetilde{T}\wedge\tau_{n} }
\left\| \int\limits_{0}^{t} U^{v}(t,s)A_{v}(s)\int\limits_{s}^{t} \sigma(r) 
~\txtd W(r)~\txtd s \right\|^{2}_{Z} \right]\\
& \leq C \mathbb{E} \int \limits_{0}^{\widetilde{T}\wedge\tau_{n}}  
|| U^{v}(t,s) A_{v}(s) \int\limits_{s}^{t}
\sigma(r)~\txtd W(r)||^{2}_{Z}~\txtd s
\leq C \widetilde{T} ||\sigma ||^{2}_{\cC([0,\widetilde{T}
\wedge\tau_{n}];\mathcal{L}_{2}(H,Z))}.
\end{align*}
In summary, we have obtained
	\begin{align*}
\mathbb{E} \left[\sup\limits_{0\leq t \leq \widetilde{T}\wedge\tau_{n}}
||u(t) -u_{0}||_{Z}^{2}\right] \leq C  R^{2} + C\left( 
\widetilde{T}^{2(1-\beta)}||f||^{2}_{\cC^{\delta}([0,\widetilde{T}\wedge
\tau_{n}];X)} + \widetilde{T}||\sigma||^{2}_{\cC([0,\widetilde{T}
\wedge\tau_{n}];\mathcal{L}_{2}(H,Z))} \right),
\end{align*}
which proves the first statement. In order to prove the second one we only 
analyze the terms containing stochastic integrals, since the remaining ones are 
obvious due to Lemma \ref{sa}. We set
$$u^{\sigma}(t):= U^{v}(t,0) (J(\sigma) (t)) - \int\limits_{0}^{t} 
U^{v}(t,s)A_{v}(s) (J(\sigma)(t) - J(\sigma(s) )~\txtd s.$$
From the estimates employed in Lemma~\ref{sa} we infer that
\begin{align*}
		||u^{\sigma} ||_{\cC^{\delta}([0,\widetilde{T}\wedge\tau_{n}];Y)} 
		\leq C_{\widetilde{T}}  ||J(\sigma) ||_{\cC^{\gamma}([0,\widetilde{T}
		\wedge\tau_{n}];Z) },
\end{align*}
where $C_{\widetilde{T}}\to 0$ as $\widetilde{T}\searrow 0$.
Now, we can apply Proposition \ref{hoelderint} which gives estimates of the 
second moment of the H\"older-norm of $J(\sigma)$ and obtain
\begin{align*}
	||u^{\sigma}||_{L^{2}(\Omega; \cC^{\delta}([0,\widetilde{T}
	\wedge\tau_{n}];Y)  )} & \leq C ||J(\sigma) ||_{L^{2}(\Omega; 
	\cC^{\gamma}([0,\widetilde{T}\wedge\tau_{n}];Z))} \\
	&\leq C C_{\widetilde{T}} \widetilde{T}^{1/2} 
	||\sigma||_{\cC([0,\widetilde{T}\wedge\tau_{n}];\mathcal{L}_{2}(H,Z))}. 
\end{align*}
Consequently, we find
\begin{align*}
	||u||_{L^{2}(\Omega; \cC^{\delta}([0,\widetilde{T}\wedge\tau_{n}];Y)  )} 
	\leq C C_{\widetilde{T}}(\widetilde{T}^{1/2} ||\sigma||_{\cC([0,
	\widetilde{T}\wedge\tau_{n}];\mathcal{L}_{2}(H,Z))} + || u_{0}||_{Z} 
	+ ||f||_{\cC^{\delta}([0,\widetilde{T}];X)}).
\end{align*}
\qed
\end{proof}\\

The next step is to establish that $\Phi:\mathcal{K}\to \mathcal{K}$ is a contraction with respect to the norm in  $L^{2}(\Omega; \cC([0,\widetilde{T}\wedge\tau_{n}];Y))$. To this aim consider two solutions 
having the same initial condition 
\begin{align}
		u_{j}(t)& =U^{v_{j}}(t,0)u_{0} + U^{v_{j}}(t,0)\int\limits_{0}^{t}
		\sigma(s) ~\txtd W(s) + \int\limits_{0}^{t}U^{v_{j}}(t,s)f(s) ~\txtd  s\nonumber\\ 
		& -\int\limits_{0}^{t}U^{v_{j}}(t,s)A_{v_{1}}(s)\int\limits_{s}^{t}
		\sigma(r)~\txtd  W(r) ~\txtd  s,\qquad j\in\{1,2\}. 
\end{align} 
Taking the difference between these two solutions	entails
\begin{align*}
&	u_{1}(t)-u_{2}(t)  = (U^{v_{1}}(t,0) - U^{v_{2}}(t,0)) u_{0} 
+ (U^{v_{1}}(t,0)-U^{v_{2}}(t,0) )\int\limits_{0}^{t}\sigma(s) ~\txtd  W(s)\\
& + \int\limits_{0}^{t} (U^{v_{1}}(t,s)-U^{v_{2}}(t,s)) f(s) ~\txtd  s
	- \int\limits_{0}^{t} (U^{v_{1}}(t,s)A_{v_{1}}(s)- U^{v_{2}}(t,s) 
	A_{v_{2}}(s)) \int\limits_{s}^{t}\sigma(r) ~\txtd  W(r) ~\txtd s\\
	& =:( a^{v_{1}}(t) - a^{v_{2}}(t) ) + (b^{v_{1}}(t) - b^{v_{2}}(t) ) 
	+ (c^{v_{1}}(t) - c^{v_{2}}(t)) - (d^{v_{1}}(t) - d^{v_{2}}(t)).
\end{align*}
In order to verify the contraction property we will separately analyze 
the terms containing initial data, the deterministic drift term, and 
the stochastic/noise term(s).

\begin{lemma}
\label{incond}
Let $u_{0}\in K$ a.s. and consider $0\leq t \leq \widetilde{T}\wedge
\tau_{n}$. Then the following estimate holds:
\begin{align*}
		\mathbb{E} \left[ || a^{v_{1}} - a^{v_{2}} ||^{2}_{\cC([0,\widetilde{T}
		\wedge\tau_{n}];Y)}\right]\leq C (Rn)^{2} \widetilde{T}^{2(\nu+\beta-\alpha-1)} 
		\mathbb{E} || v_{1} -v_{2}||^{2}_{\cC([0,\widetilde{T}\wedge\tau_{n}];Y)}.
\end{align*}	
\end{lemma}

\begin{proof} 
The proof follows the deterministic quasilinear setting discussed 
in \cite[Lemma 3.1]{Yagi5}. Knowing that for $0\leq s\leq \tau \leq 
t\leq \widetilde{T}\wedge\tau_{n}$ and $u_{0}\in K$ a.s.
\begin{align*}
\frac{\partial }{\partial \tau} U^{v_{1}}(t,\tau) U^{v_{2}}(\tau,s)u_{0}
= U^{v_{1}}(t,\tau) (A_{v_{2}}(\tau) - A_{v_{1}}(\tau))  U^{v_{2}} (\tau,s)u_{0},
\end{align*}
we have
\begin{align}
\label{estu}
U^{v_{2}}(t,s)u_{0}-U^{v_{1}}(t,s)u_{0}= 
\int\limits_{s}^{t} U^{v_{1}}(t,\tau) (A_{v_{2}}(\tau) 
- A_{v_{1}}(\tau) ) U^{v_{2}}(\tau,s)u_{0}~\txtd \tau.
\end{align}
Therefore, from (\ref{estu}) using Yosida approximations $A _{v_{j,m}}$ 
for $j\in\{1,2\}$ of the generators, we obtain	for $0\leq \theta<\nu$ that
\begin{align*}
&	A^{\theta}_{v_{1,m}}(t) (U^{v_{1,n}}(t,0) - U ^{v_{2,m}}
(t,0  ))A^{-\beta}_{v_{2,m}}(0)\\
& = \int\limits_{0}^{t} A^{\theta}_{v_{1,m}}(t) U ^{v_{1,m}}
(t,s) A _{v_{1,m}}(s) (A_{v_{1,m}}(s)^{-1} - A_{v_{2,m}}(s)^{-1}) 
A_{v_{2,m}}(s) U^{v_{2,m}}(s,0) A^{-\beta}_{v_{2,m}}(0) ~\txtd s.
\end{align*}
Letting $m\to\infty$ in the previous identity we conclude
\begin{align*}
& A^{\theta}_{v_{1}}(t)(U^{v_1}(t,0)-U^{v_2}(t,0)) 
A^{-\beta}_{v_{2}}(0)\\ &=
		\int\limits_{0}^{t}A_{v_{1}}^{\theta}(t) U^{v_{1}}(t,s) 
		A^{1-\nu }_{v_{1}}(s) A^{\nu }_{v_{1}}(s)(A_{v_{1}}(s)^{-1} 
		- A_{v_{2}}(s)^{-1}) A_{v_{2}}(s) U^{v_{2}} (s,0) 
		A^{-\beta}_{v_{2}}(0)~\txtd s.
\end{align*}
Taking into account 
(\ref{an}), (\ref{nu}) and \eqref{eq:stoptime} we get
\begin{align}
		||A_{v_{1}}^{\theta}(t)  (U^{v_1}(t,0)-U^{v_2}(t,0) )A^{-\beta}_{v_2}(0)
		||_{\mathcal{L}(X)} & \leq C n \int\limits_{0}^{t} (t-s)^{\nu - \theta-1} 
		s^{\beta-1} ||v_{1}(s)-v_{2}(s)||_{Y} ~\txtd s \nonumber \\
		& \leq Cn t^{\nu - \theta+\beta -1} ||v_{1}-v_{2} ||_{\cC
		([0,\widetilde{T}\wedge\tau_{n}];Y)},
\end{align}
from which we infer that
\begin{equation}
\label{incond1}
||A^{\theta}_{v_{1}}(t) (U^{v_{1}}(t,0) - U ^{v_{2}} (t,0))u_{0} ||_{X}\leq
 C n t^{\nu - \theta+\beta-1} ||v_{1}-v_{2}||_{{\cC([0,\widetilde{T}
\wedge\tau_{n}];Y)}} ||u_{0} ||_{\beta}.
\end{equation}
Setting $\theta:=\alpha$ in (\ref{incond1}) entails
$$||A^{\alpha}_{v_{1}}(t) (U^{v_{1}}(t,0) - U ^{v_{2}} (t,0))u_{0} ||_{X}
\leq Cn t^{\nu+\beta-\alpha -1} ||v_{1}-v_{2}||_{\cC([0,\widetilde{T}\wedge 
\tau_{n}];Y)} ||u_{0} ||_{\beta},$$
which means that
\begin{align*}
||(U^{v_{1}}(t,0)- U^{v_{2}}(t,0))u_{0} ||_{Y} & \leq C n 
\widetilde{T}^{\nu+\beta-\alpha -1}  ||v_{1} - v_{2}||_{\cC([0,\widetilde{T}
			\wedge\tau_{n}];Y)} ||u_{0}||_{
			\beta}\\ & \leq C R n \widetilde{T}^{\nu+\beta-\alpha -1} 
			||v_{1}-v_{2}||_{\cC([0,\widetilde{T}\wedge\tau_{n}];Y)},
\end{align*}
since $u_{0}\in K$ a.s. Taking expectation in the previous inequality leads 
to the desired estimate. 
Recall that we assumed 
$1-\nu<\beta-\alpha$ in (A4').\qed
\end{proof}

\begin{lemma}
\label{stoch1} 
The estimate
\begin{align}
\label{1}
\mathbb{E}\left[|| b^{v_{1}}- b^{v_{2}}||^{2}_{\cC([0,\widetilde{T}
\wedge\tau_{n}];Y)} \right] \leq C n^{2}\widetilde{T}^{2(\nu+\beta-\alpha) -1 } 
||\sigma||^{2}_{\cC([0,\widetilde{T}\wedge\tau_{n}];
\mathcal{L}_{2}(H,Z))} \mathbb{E} \left[||v_{1} -v_{2} ||^{2}_{\cC
([0,\widetilde{T}\wedge\tau_{n}];Y) }\right]
\end{align}
is valid.
\end{lemma}

\begin{proof} 
Let $0< t \leq \widetilde{T}\wedge\tau_{n}$. We have
\begin{align*}
				& || (U^{v_1}(t,0)-U^{v_{2}}(t,0))\int\limits_{0}^{t}
				\sigma(s)~\txtd  W(s) ||_{L^{2}(\Omega,Y)}\\ 
				%= ||A^{\beta}_{v_{1}}(t) (U^{v_{1}}(t,0) - 
				%U^{v_{2}} (t,0) ) \int\limits_{0}^{t} \sigma(s)~\txtd W(s) 
				%||_{L^{2}(\Omega; X)}\\
%& {\color{blue}{= ||A^{\beta}_{v_{1}} (t) (U^{v_{1}}(t,0) - 
%U^{v_{2}}(t,0) ) A^{-\beta}_{v_{2}}(0) A^{\beta}_{v_{2}}(0) 
%\int\limits_{0}^{t}\sigma(s)~\txtd W(s) ||_{L^{2}(\Omega;X)}}}\\
& \leq ||A^{\alpha}_{v_{1}} (t) (U^{v_{1}}(t,0) -U^{v_{2}}(t,0) ) A^{-\beta}_{v_{2}}(0) 
||_{\mathcal{L}(X)} ||\int\limits_{0}^{t} \sigma(s)~\txtd W(s) ||_{L^{2}(\Omega;Z)}\\
& \leq C \sqrt{t}  ||A^{\alpha}_{v_{1}} (t) (U^{v_{1}}(t,0) -U^{v_{2}}(t,0) ) A^{-\beta}_{v_{2}}(0) 
||_{\mathcal{L}(X)} ||\sigma ||_{\cC([0,t];\mathcal{L}_{2}(H,Z))}.
\end{align*}
As already seen, (\ref{incond1}) entails
\begin{align*}
||A^{\alpha}_{v_{1}}(t) (U^{v_{1}}(t,0) - U ^{v_{2}}(t,0))||_{\mathcal{L}(Z,X)} 
\leq C n \widetilde{T}^{\nu-\alpha+\beta-1} ||v_{1} -v_{2}||_{\cC(
[0,\widetilde{T}\wedge\tau_{n}];Y)}.
\end{align*}
Putting all this together proves (\ref{1}).	
 \qed
\end{proof}
		
\begin{lemma}
\label{f} 
For $0\leq t \leq \widetilde{T}\wedge\tau_{n}$ we have
\begin{align*}	\mathbb{E} \left[ ||c^{v_{1}} - c^{v_{2}} 
		||^{2}_{\cC([0,\widetilde{T}\wedge\tau_{n}];Y)}
 \right]\leq Cn^{2} \widetilde{T}^{2 (\nu-\alpha)} 
||f ||^{2}_{\cC^{\delta}([0,\widetilde{T}\wedge\tau_{n}];X)} 
\mathbb{E}\left[||v_{1} -v_{2}||^{2}_{\cC([0,\widetilde{T}
\wedge\tau_{n}];Y)}\right].
 \end{align*}
\end{lemma}

\begin{proof}
Consider $0\leq\theta<\nu\leq 1$. 
Using again Yosida approximations we infer that
\begin{align*}
& A^{\theta}_{v_{1}}(t) \int\limits_{0}^{t} ( U^{v_{1}}(t,s)- 
U^{v_{2}}(t,s)) f(s)~\txtd s\\
		& = \int\limits_{0}^{t} \int\limits_{s}^{t} A^{\theta}_{v_{1}}(t) 
		U^{v_{1}}(t,r) A_{v_{1}}(r) [A_{v_{1}}^{-1} (r) - 
		A_{v_{2}}^{-1}(r) ] A_{v_{2}}(r) U^{v_{2}}(r,s) 
		f(s)~\txtd r ~\txtd s\\
		& = \int\limits_{0}^{t} A^{\theta}_{v_{1}}(t) U^{v_{1}}(t,r) 
		A_{v_{1}}(r) [A_{v_{1}}^{-1}(r) - A_{v_{2}}^{-1}(r)] 
		A_{v_{2}}(r) \int\limits_{0}^{r} U^{v_{2}}(r,s)
		f(s)~\txtd s ~\txtd r\\
		&= \int\limits_{0}^{t} A^{\theta}_{v_{1}}(t) U^{v_{1}}(t,r) 
		A^{1-\nu}_{v_{1}}(r) A^{\nu}_{v_{1}}(r) [A_{v_{1}}^{-1}(r) 
		- A_{v_{2}}^{-1}(r)]  A_{v_{2}}(r)\int\limits_{0}^{r} 
		U^{v_{2}}(r,s)f(s)~\txtd s ~\txtd r.
\end{align*}
Regarding that 
\begin{align*}
\left\| A_{v_{2}}(r) \int\limits_{0}^{r}U^{v_{2}}(r,s) f(s)~\txtd s \right\|_{X} \leq C_{\delta} ||f||_{\cC^{\delta}([0,\widetilde{T}\wedge\tau_{n}];X)}
\end{align*}
we conclude similar to the proof of Lemma \ref{sa} 
\begin{align}
&	\left\|A^{\theta}_{v_{1}}(t) \int\limits_{0}^{t} ( U^{v_{1}}(t,s)
- U^{v_{2}}(t,s)) f(s)~\txtd s\right\|_{X}\nonumber \\
	& \leq 	C n \int\limits_{0}^{t} (t-r)^{\nu-\theta-1} ||v_{1}(r) 
	-v_{2}(r)||_{Y}  ~\txtd r\nonumber\\
	& \leq C n \int\limits_{0}^{t} (t-r)^{\nu -\theta-1} 
	 ~\txtd r ||v_{1}-v_{2} ||_{\cC([0,\widetilde{T}
	\wedge\tau_{n}];Y) } ||f||_{\cC^{\delta}([0,\widetilde{T}\wedge\tau_{n}];X)}\nonumber\\
		& \leq C n \widetilde{T}^{\nu-\theta} ||v_{1}-v_{2} ||_{\cC
		([0,\widetilde{T}\wedge \tau_{n}];Y)} || f||_{\cC^{\delta}([0,\widetilde{T}
		\wedge\tau_{n}];X)}.\label{theta:f}
\end{align}
Obviously, one also obtains
\begin{align*}
&	\left\|A^{\theta}_{v_{1}}(t) \int\limits_{0}^{t} ( U^{v_{1}}(t,s)- 
U^{v_{2}}(t,s)) f(s)~\txtd s\right\|_{X} \leq C n t^{\nu-\theta} 
|| v_{1} -v_{2}||_{\cC([0,\widetilde{T}\wedge\tau_{n}];Z)} ||f ||_{
\cC^{\delta}([0,\widetilde{T}\wedge\tau_{n}];X)}.
\end{align*}
Again setting $\theta:=\alpha$ in (\ref{theta:f}) leads to
		\begin{align*}
		\left\|\int\limits_{0}^{t} ( U^{v_{1}}(t,s)- U^{v_{2}}(t,s)) 
		f(s)~\txtd s \right\|_{Y} & \leq  C n t^{\nu-\alpha} 
		||f ||_{\cC^{\delta}([0,\widetilde{T}\wedge\tau_{n}];X)} ||v_{1}-v_{2} 
		||_{\cC([0,\widetilde{T}\wedge \tau_{n}];Y)}.
\end{align*} 	
Taking expectation yields the  claimed result.
\qed 
\end{proof}
\begin{remark}
Instead of taking $f$ H\"older-continuous with values in $X$ one can let $f$ be just continuous with values in  $X_{\hat\rho}$, for a suitable chosen $\hat\rho>0$, consult \cite{Yagi1} and the 
references specified therein.
\end{remark}

We now analyze the generalized 
stochastic convolution. To this aim the higher space-regularity of $\sigma$ 
is required. Such a condition is natural for this technique, since one needs additional regularity assumptions when building the difference of two evolution systems, compare for instance Lemma \ref{incond}. 

\begin{lemma}
\label{stochconv}
We have 
\begin{align*}
\mathbb{E} \left[|| d^{v_{1}} -d^{v_{2}}||_{\cC([0,\widetilde{T}
\wedge\tau_{n}];Y)}^{2} \right]\leq  Cn^{2}\widetilde{T}^{2 (\nu -\alpha+2\beta)  -1} 
||\sigma ||^{2}_{\cC([0,\widetilde{T}\wedge\tau_{n}];\mathcal{L}_{2}(H,X_{2\beta}))}
\mathbb{E}\left[||v_{1}-v_{2} ||^{2}_{\cC([0,\widetilde{T}\wedge\tau_{n}];Y)}\right].
\end{align*}
\end{lemma}

\begin{proof}
Let $0\leq t \leq \widetilde{T}\wedge\tau_{n}$. Then
	\begin{align} 
	\label{stoch}
&	\int\limits_{0}^{t}(U^{v_{1}}(t,s)A_{v_{1}}(s) - U^{v_{2}}(t,s)
A_{v_{2}}(s)) \int\limits_{s}^{t}\sigma(r)~\txtd W(r)~\txtd s \nonumber \\
& = \int\limits_{0}^{t} (U^{v_{1}}(t,s)- U^{v_{2}}(t,s))  A_{v_{1}}(s) 
\int\limits_{s}^{t}\sigma(r)~\txtd W(r) ~\txtd s \nonumber\\
&	 +\int\limits_{0}^{t} U^{v_{2}}(t,s) (A_{v_{1}}(s) - A_{v_{2}}(s)) 
	\int\limits_{s}^{t}\sigma(r)~\txtd W(r) ~\txtd s\nonumber\\
& =: (e^{v_{1}}(t) - e^{v_{2}} (t) ) + (f^{v_{1}}(t) - f^{v_{2}}(t) ).
\end{align}
Recalling Lemma \ref{incond} we rewrite the first term as
\begin{align*}
&	\int\limits_{0}^{t} (U^{v_{1}}(t,s)-U^{v_{2}}(t,s )) 
A_{v_{1}}(s)\int\limits_{s}^{t} \sigma(r)~\txtd W(r)~\txtd s\\
%&{\color{blue}{ =\int\limits_{0}^{t} (U^{v_{1}}(t,s) - U^{v_{2}}(t,s))
%A_{v_{1}}(s) A^{-1-\beta}_{v_{1}}(s) A^{1+\beta}_{v_{1}}(s) 
%\int\limits_{s}^{t} \sigma(r)dW(r)ds}}\\
%& {\color{blue}= \int\limits_{0}^{t} \int\limits_{s}^{t} U^{v_{2}}
%(t,q) (A_{v_{2}}(q) - A_{v_{1}}(q) ) U^{v_{1}}(q,s) 
%A^{-\beta}_{v_{1}}(s)dq A^{1+\beta}_{v_{1}}(s) 
%\int\limits_{s}^{t} \sigma(r)dW(r) ds}\\
 & = \int\limits_{0}^{t}  \int\limits_{s}^{t} U^{v_{2}}(t,q) 
A_{v_{2}}(q) (A^{-1}_{v_{1}}(q)- A^{-1}_{v_{2}}(q) ) 
A_{v_{1}}(q) U^{v_{1}}(q,s) A^{1-2\beta}_{v_{1}}(s)~\txtd 
q ~A^{2\beta}_{v_{1}}(s) \int\limits_{s}^{t} \sigma(r)~\txtd 
W(r) ~\txtd s.
 \end{align*}
Let $0\leq \theta<\nu$. Regarding that
\begin{align*}
&\left\|\int\limits_{0}^{t} (U^{v_{1}}(t,s) 
- U^{v_{2}}(t,s )) A_{v_{1}}(s) \int\limits_{s}^{t}\sigma(r)~\txtd 
W(r) ~\txtd s\right\|\\
&= \Bigg\|\int\limits_{0}^{t} \int\limits_{s}^{t}  U^{v_{2}}(t,\tau) A_{v_{2}}(\tau) (A_{v_{1}}^{-1}(\tau) - A_{v_{2}}^{-1} (\tau) ) A_{v_{1}}(\tau) U^{v_{1}}(\tau,s) A_{v_{1}}(s) A^{-2\beta}_{v_{1}}(s) d\tau A^{2\beta}_{v_{1}}(s) \int\limits_{s}^{t} \sigma(r)dW(r) ds \Bigg\|\\
\end{align*}
 we get
\begin{align*}
 &\left\|\int\limits_{0}^{t} A^{\theta}_{v_{1}}(t) (U^{v_{1}}(t,s) 
- U^{v_{2}}(t,s )) A_{v_{1}}(s) \int\limits_{s}^{t}\sigma(r)~\txtd 
W(r) ~\txtd s\right\|_{L^{2}(\Omega;X)}\\
 	 & \leq C n \int \limits_{0}^{t} (t-s) ^{\nu -\theta+2\beta +1/2-2}
	~\txtd s ||\sigma ||_{\cC([0,t];\mathcal{L}_{2}(H, X_{2\beta}))} 
	|| v_{1}-v_{2}||_{ \cC([0,t];Y) }\\
 	 & \leq C nt^{\nu-\theta+2\beta -1/2} ||\sigma ||_{\cC([0,t];
	\mathcal{L}_{2}(H, X_{2\beta}))} || v_{1}-v_{2}||_{ \cC([0,t];Y) }.
\end{align*}

Consequently,
\begin{align}
 & \mathbb{E} \left[\sup\limits_{0\leq t \leq \widetilde{T}\wedge
\tau_{n} } \left \| \int\limits_{0}^{t} A^{\theta}_{v_{1}}(t)(
U^{v_{1}}(t,s)-U^{v_{2}}(t,s )) A_{v_{1}}(s)\int\limits_{s}^{t} 
\sigma(r)~\txtd W(r)~\txtd s\right\|^{2}_{X} \right]\nonumber \\ 
 &  \leq C \mathbb{E} \left[\int \limits_{0}^{\widetilde{T}\wedge
\tau_{n} } || A^{\theta}_{v_{1}}(t)(U^{v_{1}}(t,s)-U^{v_{2}}(t,s )) 
A_{v_{1}}(s)\int\limits_{s}^{t} \sigma(r)~\txtd W(r)||^{2}_{X}~
\txtd s\right] \nonumber \\
  & \leq C n^{2}\widetilde{T} ^{2 (\nu -\theta+2\beta) -1}  
	||\sigma ||^{2}_{\cC([0,\widetilde{T}\wedge\tau_{n}];\mathcal{L}_{2}
	(H,X_{2\beta}))} \mathbb{E}\left[||v_{1} -v_{2}||^{2}_{\cC([
	0,\widetilde{T}\wedge\tau_{n}];Y ) }\right]. \label{sigma1}
\end{align} 
 	
Using that $A^{1-2\beta}_{v_{1}}\in\mathcal{L}(X)$,
%{\color{blue}{	\begin{align*}
% &	\int\limits_{0}^{t} U^{v_{2}}(t,s) (A_{v_{1}}(s) 
%- A_{v_{2}}(s)) \int\limits_{s}^{t}\sigma(r)~\txtd W(r) ~\txtd s\\
% & = \int \limits_{0} ^{t} U^{v_{2}}(t,s) A^{1-\nu}_{v_{2}}
%(s) A^{\nu}_{v_{2}}(s)  (A^{-1}_{v_{2}}(s) -A^{-1}_{v_{1}}(s )) 
%A^{-\beta }_{v_{1}}(s) A^{1+\beta}_{v_{1}}(s) \int\limits_{s}^{t} 
%\sigma(r) ~\txtd  W(r)~\txtd s,
%\end{align*} }}
we estimate the second term in (\ref{stoch}) as
\begin{align*}
 	 &\left\|\int\limits_{0}^{t}A^{\theta}_{v_{2}}(t) U^{v_{2}}(t,s) 
	(A_{v_{1}}(s) - A_{v_{2}}(s)) \int\limits_{s}^{t}\sigma(r)~\txtd 
	W(r) ~\txtd s\right\|_{L^{2}(\Omega;X)} \\
% 	 & {\color{blue}{ \leq Cn  \int \limits_{0}^{t} (t-s)^{\nu-\theta  
%	-1} ||\int\limits_{s}^{t}\sigma(r) ~\txtd W(r)||_{L^{2}(\Omega; 
%	X_{1+\beta})  } ~\txtd s ||v_{1}-v_{2} ||_{\cC([0,t ];Z)} }}\\
 	 & \leq C n t^{\nu- \theta +1/2} ||\sigma||_{\cC([0,t];\mathcal{L}_{2}
	(H, X_{2\beta})) } ||v_{1}-v_{2}||_{\cC([0,t];Y) }.
\end{align*}
Summarizing the previous calculations, we obtain
\begin{align}
\label{sigma2}
 & \mathbb{E}\left[\sup\limits_{0\leq t \leq \widetilde{T}
\wedge\tau_{n}} \left\|\int\limits_{0}^{t}A^{\theta}_{v_{2}}(t) 
U^{v_{2}}(t,s) (A_{v_{1}}(s) - A_{v_{2}}(s)) \int\limits_{s}^{t}
\sigma(r)~\txtd W(r) ~\txtd s\right\|^{2}_{X}\right] \\
  & \leq C n^{2} \widetilde{T}^{2(\nu-\theta)+1} ||\sigma||^{2}_{
	\cC([0,\widetilde{T}\wedge\tau_{n}];\mathcal{L}_{2}(H, 
	X_{2\beta}))  } \mathbb{E}\left[||v_{1} -v_{2}||^{2}_{
	\cC([0,\widetilde{T}\wedge\tau_{n}];Y ) }\right] \nonumber.
\end{align}
Setting $\theta:=\alpha$ in (\ref{sigma1}) and (\ref{sigma2}), the 
triangle-inequality in $L^{2}(\Omega;Y)$ proves the statement.
 \qed
\end{proof}\\

Collecting all these results finally yields.

\begin{lemma}\label{Phi}
The mapping $\Phi:\cK\ra \cK$ is a contraction for a sufficiently 
small $\widetilde{T}$.
\end{lemma}

\begin{proof}
We obtained that
\begin{align*}
& \mathbb{E}\left[||u_{1} - u_{2}||^{2}_{\cC([0,\widetilde{T}
\wedge\tau_{n}];Y)}\right] \leq  C (Rn)^{2} \widetilde{T}^{2(
\nu+\beta-\alpha-1)} \mathbb{E} \left[|| v_{1} -v_{2}||^{2}_{\cC([0,\widetilde{T}\wedge\tau_{n}];Y)}\right]\\
& + C n^{2}\widetilde{T}^{2(\nu-\alpha+\beta) -1 } ||\sigma||^{2}_{
\cC([0,\widetilde{T}\wedge\tau_{n}];\mathcal{L}_{2}(H,Z))} 
\mathbb{E} \left[||v_{1} -v_{2} ||^{2}_{\cC([0,
\widetilde{T}\wedge\tau_{n}];Y) }\right]\\
& + Cn^{2} \widetilde{T}^{2 (\nu-\alpha)} ||f ||^{2}_{
\cC([0,\widetilde{T}\wedge\tau_{n}];X)} \mathbb{E}\left[
||v_{1} -v_{2}||^{2}_{\cC([0,\widetilde{T}\wedge\tau_{n}];
Y)}\right]\\
& +  C n^{2} \widetilde{T} ^{2(\nu -\alpha+2\beta) -1} ||\sigma ||^{2}_{\cC([
0,\widetilde{T}\wedge\tau_{n}];\mathcal{L}_{2}(H,X_{2\beta}))} 
\mathbb{E}	\left[||v_{1} -v_{2}|| ^{2}_{ \cC([0,\widetilde{T}
\wedge\tau_{n}];Y) }\right].
\end{align*}
Therefore, choosing $\widetilde{T}$ small enough we obtain that $\Phi$ is a contraction.
\qed
\end{proof}\\

 In conclusion, due to Lemma \ref{Phi}, Banach's fixed-point theorem proves Theorem \ref{sol}.\\

We now derive analogously to the proof of Theorem 1.3 in \cite{Kim1} assertions regarding the corresponding stopping times.
\begin{lemma}\label{ps1}\emph{(\textbf{Positivity of the stopping times})}
	Let $0<\varepsilon<1$ be fixed. Then under the assumptions of Theorem \ref{sol} we have
	\begin{align}
	\mathbb{P} (\tau_{n}>\varepsilon) > 1 -\varepsilon^{2} \Big(\widetilde{C}\mathbb{E}||u_{0}||^{2}_{Z} + \widetilde{C} \varepsilon||\sigma||^{2}_{\cC([0,\varepsilon];\mathcal{L}_{2}(H,Z))} + \widetilde{C} \varepsilon^{2(1-\beta)}||f||^{2}_{\cC^{\delta}([0,\varepsilon];X)}\Big),
	\end{align}
	where the positive constant $\widetilde{C}=\widetilde{C}(\delta)$ is independent of $\varepsilon$ and $u_{0}$.
\end{lemma}	
\begin{proof}
For $0<\varepsilon<1$ there exists a positive number $n$ such that
\begin{align*}
\frac{1}{n+1}\leq \varepsilon<\frac{1}{n}.
\end{align*}
Applying Theorem \ref{sol} we obtain that $(u,\widetilde{T}\wedge\tau_{n})$ is the pathwise mild solution of (\ref{cp1}). Our aim is to derive estimates for
\begin{align}\label{e}
\mathbb{E} \sup\limits_{0\leq t \leq \varepsilon} ||u(t\wedge\tau_{n})||^{2}_{Z}.
\end{align}
Using the definition of $\tau_{n}$ we know that
\begin{align}\label{fuerp}
\left\{\omega:\sup\limits_{0\leq t \leq \varepsilon} ||u(t\wedge\tau_{n})||_{Z} <n \right\} \subset \{\omega:\tau_{n}(\omega) >\varepsilon\}.
\end{align}
Therefore, if we show (\ref{e}), Chebyshev's inequality proves the statement. \\

By the same computation as performed in Lemma \ref{meansquare} we have
\begin{align*}
\mathbb{E} \sup\limits_{0\leq t \leq \widetilde{T}\wedge\tau_{n}} ||U^{u}(t,0) u_{0}||^{2}_{Z} \leq \widetilde{C} e ^{n^{2/\delta}\widetilde{T}} \mathbb{E}||u_{0}||^{2}_{Z}.
\end{align*}
Furthermore,
\begin{align*}
\mathbb{E}\sup\limits_{0\leq t \leq \widetilde{T}\wedge\tau_{n} } ||U^{u}(t,0) \int\limits_{0}^{t} \sigma(r)~\txtd W(r) ||^{2}_{Z} \leq  \widetilde{C} e ^{n^{2/\delta}\widetilde{T}} \widetilde{T} ||\sigma||^{2}_{\cC([0,\widetilde{T}];\mathcal{L}_{2}(H,Z))}
\end{align*}
and 

\begin{align*}
&\mathbb{E}\left[\sup\limits_{0\leq t \leq \widetilde{T}\wedge\tau_{n} }
\left\| \int\limits_{0}^{t} U^{v}(t,s)A_{v}(s)\int\limits_{s}^{t} \sigma(r) 
~\txtd W(r)~\txtd s \right\|^{2}_{Z} \right]\\
& \leq C \mathbb{E} \int \limits_{0}^{\widetilde{T}\wedge\tau_{n}}  
|| U^{v}(t,s) A_{v}(s) \int\limits_{s}^{t}
\sigma(r)~\txtd W(r)||^{2}_{Z}~\txtd s
\leq \widetilde{C} e ^{n^{2/\delta}\widetilde{T}} \widetilde{T} ||\sigma ||^{2}_{\cC([0,\widetilde{T}];\mathcal{L}_{2}(H,Z))}.
\end{align*}
This means that
\begin{align*}
\mathbb{E} \sup\limits_{0\leq t \leq \widetilde{T}\wedge\tau_{n}} ||u(t)||^{2}_{Z} &\leq \widetilde{C} e ^{n^{2/\delta}\widetilde{T}} \left(\mathbb{E} ||u_{0}||^{2}_{Z} +  \widetilde{T}||\sigma||^{2}_{\cC([0,\widetilde{T}];\mathcal{L}_{2}(H,Z))}
 +  \widetilde{T}^{2(1-\beta)} ||f||^{2}_{\cC^{\delta}([0,\widetilde{T}];X)} \right),
\end{align*}
which implies

\begin{align}\label{b:stopping}
\mathbb{E} \sup\limits_{0\leq t \leq \varepsilon} ||u(t\wedge\tau_{n})||^{2}_{Z} &\leq \widetilde{C} e ^{n^{2/\delta}\varepsilon} \left( \mathbb{E} ||u_{0}||^{2}_{Z} + \varepsilon ||\sigma||^{2}_{\cC([0,\varepsilon];\mathcal{L}_{2}(H,Z))}  +   \varepsilon^{2(1-\beta)} ||f||^{2}_{\cC^{\delta}([0,\varepsilon];X)} \right).
\end{align}

From (\ref{fuerp}) and Chebyshev's inequality we have that
\begin{align*}
\mathbb{P} (\tau_{n}>\varepsilon)& >\mathbb{P} (\sup\limits_{0\leq t \leq \varepsilon} ||u(t\wedge\tau_{n})||_{Z} <n ) \geq 1- \frac{1}{n^{2}} \mathbb{E} \sup\limits_{0\leq t \leq\varepsilon} ||u(t\wedge\tau_{n})||^{2}_{Z}\\
& \geq 1 - \widetilde{C}\frac{e^{\varepsilon n^2/\delta}}{n^{2}} \left( \mathbb{E}||u_{0}||^{2}_{Z} + \varepsilon||\sigma||^{2}_{\cC([0,\varepsilon];\mathcal{L}_{2}(H,Z))} + \varepsilon^{2(1-\beta)}||f||^{2}_{\cC^{\delta}([0,\varepsilon];X)} \right)\\
&  \geq 1 -\varepsilon^{2} \Big(\widetilde{C}\mathbb{E}||u_{0}||^{2}_{Z} + \widetilde{C} \varepsilon||\sigma||^{2}_{\cC([0,\varepsilon];\mathcal{L}_{2}(H,Z))} + \widetilde{C} \varepsilon^{2(1-\beta)}||f||^{2}_{\cC^{\delta}([0,\varepsilon];X)}\Big).
\end{align*}
	\qed
	\end{proof}
		
%%\begin{remark}
%%One can show that $u$ is a maximal local pathwise mild solution 
%%of (\ref{cp1}). We provide the corresponding proof of this statement 
%%in the semilinear case at the end of this section.
%%00a\end{remark}

As already discussed, the next step is to extend the results established 
in Theorem \ref{sol} and include nonlinearities of semilinear type. More 
precisely, we consider
\begin{equation}\label{cp2}
\begin{cases}
	~\txtd u(t)= \left[(Au(t))(u(t)) + F(t,u(t))\right] ~\txtd t 
	+ \sigma(t,u(t)) ~\txtd W(t), ~ t\in[0,T]\\
	u(0)=u_{0} \in K \mbox{ a.s.}
\end{cases}
\end{equation}
and make standard local Lipschitz and growth assumptions on 
$F:\Omega\times [0,T] \times X \to X$ and $\sigma:\Omega\times [0,T] \times X\to\mathcal{L}_{2}(H,X_{2\beta})$.
In particular, we assume that there exist constants $L_{F}=L_{F}(n), l_{F}=l_{F}(n), 
L_{\sigma}=L_{\sigma}(n), l_{\sigma}=l_{\sigma}(n)>0$ such that 
\begin{equation}\label{lip:f}
||F(u)-F(v) ||_{X} \leq  L_{F}|| u-v||_{X} , 
\mbox{  } ||u||_{Z}\leq n , ||v||_{Z}\leq n,
\end{equation}
\begin{equation} \label{growth:f}
|| F(u)||_{X}\leq l_{F} (1+ ||u||_{X}), \mbox{ } 
||u||_{Z}\leq n,
\end{equation}
and respectively
\begin{equation}\label{lip:sigma}
||\sigma(u) - \sigma(v) ||_{\mathcal{L}_{2} (H,X_{2\beta})}
\leq  L_{\sigma}|| u-v||_{X}, \mbox{ }||u||_{Z}\leq n, ||v||_{Z}\leq n ,
\end{equation}
\begin{equation}\label{growth:sigma}
||\sigma(u)||_{\mathcal{L}_{2}(H,X_{2\beta})} \leq l_{\sigma}
 (1+||u||_{X} ), \mbox{ } ||u||_{Z}\leq n.
\end{equation}

We remark that since $\mathcal{L}_{2}(H, X_{2\beta})\hookrightarrow 
\mathcal{L}_{2}(H,X_{\beta})$ and $Y\hookrightarrow X$ we also get
\begin{align}
||\sigma(u)-\sigma(v) ||_{\mathcal{L}_{2}(H,X_{\beta}) }  
\leq C L_{\sigma} ||\sigma(u)-\sigma(v) ||_{\mathcal{L}_{2}
	(H,X_{2\beta}) } \leq CL_{\sigma}  || u-v ||_{X}\leq CL_{\sigma} 
|| u-v||_{Y}.
\end{align}
We can now state our main result of this work.

\begin{theorem}
\label{multiplicative}
The quasilinear SPDE (\ref{cp2}) possesses a unique local pathwise 
mild solution $u\in L^{0}(\Omega;\cB([0,\widetilde{T}\wedge\tau_{n}];Z))
\cap L^{0}(\Omega; \cC^{\delta}([0,\widetilde{T}\wedge\tau_{n}];Y))$ 
given by
\begin{align}
u(t)&=U^{u}(t,0) u_{0} +  U^{u}(t,0) \int\limits_{0}^{t}
\sigma(r,u(r))~\txtd W(r) + \int\limits_{0}^{t} U^{u}(t,s)
f(s,u(s))~\txtd s\nonumber\\
& -\int\limits_{0}^{t} U^{u}(t,s) A(u(s)) \int\limits_{s}^{t} 
\sigma(r,u(r))~\txtd W(r)~\txtd s\label{u2}.
\end{align}
\end{theorem}

%%Here, $\{U^{u}(t,s)\}_{\{0\leq s\leq t \leq \widetilde{T}
%%\wedge\tau_{n}\}}$ denotes the evolution system for the family 
%%of sectorial operators $\{A(u(t))\}_{\{0\leq t \leq
%%\widetilde{T}\wedge\tau_{n}\}}$. 
We note that all the results from the linear case concerning the definition and regularity properties 
of the generalized stochastic convolution can be extended to the 
nonlinear setting as discussed in \cite[Section 5.1]{PronkVeraar}.\medskip

In order to prove Theorem \ref{multiplicative}, analogously to the 
linear case, we first let $v\in\mathcal{K}$ a.s.,~set 
$f_{v}(t):=F(t,v(t))$, $\sigma_{v}(t):=\sigma(t,v(t))$ and consider
the Cauchy problem
\begin{equation}\label{qv}
\begin{cases}
~\txtd u(t) =\left[(Au(t))(u(t)) + f_{v}(t)\right]~\txtd t 
+ \sigma_{v}(t) ~\txtd  W(t), ~ t\in[0,\widetilde{T}]\\
u(0)=u_{0}\in K \mbox{  a.s}.
\end{cases}
\end{equation}
Note that all the assumptions of Theorem \ref{sol} are satisfied. This means that the quasilinear inhomogenuous 
equation (\ref{qv}) possesses a unique pathwise mild solution 
$u\in L^{0}(\Omega;\cB([0,\widetilde{T}\wedge\tau_{n}];Z)) 
\cap L^{0}(\Omega; \cC^{\delta}([0,\widetilde{T}\wedge\tau_{n}];Y))$ 
such that
\begin{align*}
u(t)&=U^{u}(t,0)u_{0} + U^{u}(t,0)\int\limits_{0}^{t}
\sigma_{v}(r)~\txtd W(r) + \int\limits_{0}^{t}U^{u}(t,s)
f_{v}(s)~\txtd s\\ &-\int\limits_{0}^{t}U^{u}(t,s)A_{u}(s) 
\int\limits_{s}^{t}\sigma_{v}(r)~\txtd W(r)~\txtd s.
\end{align*}
In order to obtain a solution for (\ref{cp2}) by a fixed-point 
argument, we define just as before the mapping 
$$\Phi(v):=u, \mbox{ for } v\in \mathcal{K}. $$ 
One can show analogously to the proof of Lemma \ref{sa} that 
this maps $\mathcal{K}$ into itself if one chooses 
$\widetilde{T}$ small enough.\medskip

We now verify the contraction property with respect to the norm in $L^{2}(\Omega; \cC([0,\widetilde{T}\wedge\tau_{n}];Y))$. The computation relies 
on similar estimates as for (\ref{v}) combined with the local Lipschitz 
continuity and growth boundedness of $F$ and $\sigma$. 

\begin{lemma} 
The mapping $\Phi$ is a contraction if $\widetilde{T}$ is 
sufficiently small.
\end{lemma}

\begin{proof} 
Let $0<t \leq \widetilde{T}\wedge\tau_{n}$. Considering 
the difference between two solutions yields
\begin{align}
&u_{1}(t)-u_{2}(t)  = (U^{u_{1}}(t,0) -U^{u_{2}}(t,0 ))u_{0} 
+  (U^{u_{1}}(t,0) -U^{u_{2}}(t,0  )) \int\limits_{0}^{t}
\sigma_{v_{1}}(r)~\txtd W(r)\nonumber\\
& + U^{u_{2}}(t,0)\int\limits_{0}^{t}(\sigma_{v_{1}}(r)-
\sigma_{v_{2}}(r))~\txtd W(r) + \int\limits_{0}^{t}(
U^{u_{1}}(t,s) -U^{u_{2}}(t,s ))f_{v_{1}}(s)~\txtd s\nonumber\\
& + \int\limits_{0}^{t} U^{u_{2}}(t,s)(f_{v_{1}}(s)- 
f_{v_{2}}(s))~\txtd s - \int\limits_{0}^{t} U^{u_{2}}(t,s)
A_{u_{2}}(s)\int\limits_{s}^{t}(\sigma_{v_{1}}(r)-\sigma_{v_{2}}
(r)) ~\txtd  W(r)~\txtd s\nonumber\\
& -\int\limits_{0}^{t} (U^{u_{1}}(t,s) - U^{u_{2}}(t,s)) 
A_{u_{1}}(s)  \int\limits_{s}^{t}\sigma_{v_{1}}(r)~\txtd 
W(r)~\txtd s\nonumber\\
&- \int\limits_{0}^{t} U^{u_{2}}(t,s) (A_{u_{1}}(s) -
A_{u_{1}}(s)) \int\limits_{s}^{t} \sigma_{v_{1}}(r)~\txtd 
W(r)~\txtd s\nonumber\\
& =: (\widehat{a}^{u_{1}}(t) -\widehat{a}^{u_{2}}(t) ) + 
(\widehat{b}^{u_{1}}(t) - \widehat{b}^{u_{2}}(t))  + 
(\widehat{c}^{v_{1}}(t) - \widehat{c}^{v_{2}}(t) ) + 
(\widehat{d}^{u_{1}}(t) - \widehat{d}^{u_{2}}(t) )\nonumber\\
&+ (\widehat{e}^{v_{1}}(t) -\widehat{e}^{v_{2}}(t) ) + 
(\widehat{f}^{v_{1}}(t) -\widehat{f}^{v_{2}}(t) ) + 
(\widehat{g}^{u_{1}}(t) -\widehat{g}^{u_{2}}(t) ) + 
(\widehat{h}^{u_{1}}(t) - \widehat{h}^{u_{2}}(t)).\label{eq:defah}
\end{align}	
We now provide suitable estimates for each of the terms above in 
appropriate function spaces. For the first one, as discussed in 
Lemma \ref{incond}, we have
\begin{align*}
\mathbb{E}\left[ || \widehat{a}^{u_{1}} - \widehat{a}^{u_{2}}||^{2}_{
\cC([0,\widetilde{T}\wedge\tau_{n}];Y) }\right]\leq C (Rn)^{2}
 \widetilde{T}^{2(\nu +\beta-\alpha -1)}  \mathbb{E}\left[|| u_{1} -
u_{2}|| ^{2}_{\cC([0,\widetilde{T}\wedge\tau_{n}];Y) }\right].
\end{align*}
From Lemma \ref{stoch1}, applying the Burkholder-Davis-Gundy inequality 
and (\ref{growth:sigma}), the second term yields
\begin{align*}
& \mathbb{E}\left[||\widehat{b}^{u_{1}} - \widehat{b}^{u_{2}}||^{2}_{
\cC([0,\widetilde{T}\wedge\tau_{n}];Y)}\right]= \mathbb{E} \left[\sup\limits_{
0\leq t \leq \widetilde{T}\wedge\tau_{n } }|| A^{\alpha}_{v_{1}}(t)(U^{u_{1}}(t,0) -
U^{u_{2}}(t,0  )) \int\limits_{0}^{t}\sigma_{v_{1}}(r)~\txtd W(r) 
||^{2}_{Y}\right]\\
& \leq C \mathbb{E} \left[\sup\limits_{0\leq t \leq \widetilde{T}
\wedge\tau_{n} }
 ||A^{\alpha}_{v_{1}}(t) (U^{u_{1}}(t,0) - U^{u_{2}}(t,0) ) A^{-\beta}_{v_{2}}(0)
||^{2}_{\mathcal{L}(X)} ||\int\limits_{0}^{t}\sigma_{v_{1}}(s)
 ~\txtd W(s) ||^{2}_{Z}\right] \\
& \leq C n^{2} \widetilde{T}^{2(\nu -\alpha+\beta)-2} ||u_{1} -u_{2} 
||^{2}_{\cC([0,\widetilde{T}\wedge\tau_{n}];Y)} 
\mathbb{E}\left[ \int\limits_{0}^{\widetilde{T}\wedge\tau_{n}} 
||\sigma_{v_{1}}(t) ||^{2}_{\mathcal{L}_{2}(H,Z)}~\txtd t\right]\\
	& \leq C (l_{\sigma} n)^{2} \widetilde{T}^{2(\nu -\alpha+\beta)-1} 
	||u_{1} -u_{2} ||^{2}_{\cC([0,\widetilde{T}\wedge
	\tau_{n}];Y)} \mathbb{E} \left[\sup\limits_{0\leq t \leq
	\widetilde{T}\wedge\tau_{n}} (1+||v_{1}(t) ||^{2}_{Z})\right]\\
& \leq  C (l_{\sigma} n)^{2} \widetilde{T}^{2(\nu -\alpha+\beta)-1} 
||u_{1} -u_{2} ||^{2}_{\cC([0,\widetilde{T}\wedge
\tau_{n}];Y)} \left(1+ R^{2} + \mathbb{E}\left[||u_{0}||^{2}_{Z}\right]\right)\\
& \leq C (C_{R}l_{\sigma}n)^{2} \widetilde{T}^{2(\nu -\alpha+\beta) -1} 
\mathbb{E}\left[||u_{1}-u_{2} ||^{2}_{\cC([0,\widetilde{T}
\wedge\tau_{n}];Y) }\right]. 
\end{align*} 
According to (\ref{lip:sigma}) we get
\begin{align*}
&\mathbb{E} \left[ || \widehat{c}^{v_{1}} -\widehat{c}^{v_{2}} 
||_{\cC([0,\widetilde{T}\wedge\tau_{n}];Y)}^{2} \right]
=\mathbb{E}\left[ \sup\limits_{0\leq t \leq \widetilde{T }
 \wedge\tau_{n}}\left\|U^{u_{2}}(t,0) \int\limits_{0}^{t}
(\sigma_{v_{1}}(r)-\sigma_{v_{2}}(r))~\txtd W(r) \right\|_{Y}^{2}\right]\\
%& {\color{blue}{\leq C\mathbb{E} \sup\limits_{0\leq t \leq 
%\widetilde{T}\wedge\tau_{n} } ||A^{\beta}_{u_{2}}(t)
% U^{u_{2}}(t,0) ||^{2}_{\mathcal{L}(Z,X)} ||\int\limits_{0}^{t}
% (\sigma_{v_{1}}(r) -\sigma_{v_{2}}(r)) ~\txtd W(r)||^{2}_{Z} }}\\ 
& \leq C L^{2}_{\sigma}\widetilde{T}^{2(\beta-\alpha)+1} \mathbb{E}\left[  
||v_{1} -v_{2}||^{2}_{\cC([0,\widetilde{T}\wedge\tau_{n}];Y)}\right].
\end{align*}
Keeping Lemma \ref{f} in mind, together with the fact that $v_{1}\in 
\mathcal{K}$ a.s., yields due to (\ref{growth:f})
\begin{align*}
& \left\|\int\limits_{0}^{t}(U^{u_{1}}(t,s) -U^{u_{2}}(t,s ))f_{v_{1}}
(s)~\txtd s\right\|_{Y}= \left\|\int\limits_{0}^{t}A^{\alpha}_{u_{1}}
(t)(U^{u_{1}}(t,s) -U^{u_{2}}(t,s ))f_{v_{1}}(s)~\txtd s\right\|_{X}\\
& \leq C n \int\limits_{0}^{t}  (t-\tau)^{\nu -\alpha-1} ||u_{1}(\tau)
-u_{2}(\tau) ||_{Y} ||f_{v_{1}}||_{\cC^{\delta}([0,t];X)}~\txtd \tau\\
%& \leq C n l_{F} \int\limits_{0}^{t} (t-\tau) ^{\nu-\beta-1} ||
%u_{1}(\tau)-u_{2}(\tau) ||_{Z} \int\limits_{0}^{\tau} (\tau 
%-s)^{\mu -1} (1+||v_{1}(s)||_{Z} ) ~\txtd s ~\txtd\tau\\
& \leq C C_{R} n l_{F} \widetilde{T}^{\nu-\alpha} ||u_{1}-
u_{2} ||_{ \cC([0,\widetilde{T}\wedge\tau_{n}];Y) }.
\end{align*}
Consequently, we find taking the expectation that
\begin{align*}
\mathbb{E} \left[|| \widehat{d}^{u_{1}} -\widehat{d}^{u_{2}}
||^{2}_{\cC([0,\widetilde{T}\wedge\tau_{n}];Y)}\right] \leq C 
(C_{R} n l_{F} )^{2} \widetilde{T}^{2(\nu -\alpha)} 
\mathbb{E} \left[|| u_{1}-u_{2}||^{2}_{\cC([0,\widetilde{T}\wedge
\tau_{n}];Y)}\right].
\end{align*}
From (\ref{lip:f}) and using the Lipschitz assumption for $f$ we 
can directly infer, via similar calculations as above that
%\begin{align*}
%\left\|\int\limits_{0}^{t} U^{u_{2}}(t,s)(f_{v_{1}}(s)- f_{v_{2}}
%(s))~\txtd s\right\|_{Z} &= \left\|\int\limits_{0}^{t}A^{\beta}_{
%u_{2}}(t) U^{u_{2}}(t,s) A^{-\mu}_{u_{2}}(s) A^{\mu}_{u_{2}}(s)(
%f_{v_{1}}(s) - f_{v_{2}}(s) ) ~\txtd s\right\| _{X}\\
%& \leq C \int\limits_{0}^{t} (t-s)^{\mu-\beta} ||f_{v_{1}}(s) - 
%f_{v_{2}}(s) ||_{X_{\mu}}~\txtd s\\
%& \leq C L_{F} \int\limits_{0}^{t} (t-s)^{\mu -\beta} || v_{1}(s)
%-v_{2}(s) ||_{Z}~\txtd s\\
%& \leq C L_{F} \widetilde{T}^{\mu-\beta+1} ||v_{1} -v_{2}||_{\cC(
%[0,\widetilde{T}\wedge\tau_{n}];Z)}.
%\end{align*}
%Therefore,
\begin{align*}
\mathbb{E} \left[|| \widehat{e}^{v_{1}} -\widehat{e}^{v_{2}} ||^{2}_{
\cC([0,\widetilde{T}\wedge\tau_{n}];Y)}\right] \leq C L^{2}_{F} 
\widetilde{T}^{2(1-\alpha)} \mathbb{E} \left[||v_{1}-v_{2} 
||^{2}_{ \cC([0,\widetilde{T}\wedge\tau_{n}];Y)}\right].
\end{align*}

Similar computations as in the proof of Lemma \ref{meansquare} 
together with (\ref{lip:sigma}) imply 
%\begin{align*}
%& \left\| \int\limits_{0}^{t} U^{u_{2}}(t,s)A_{u_{2}}(s)\int
%\limits_{s}^{t}(\sigma_{v_{1}}(r)-\sigma_{v_{2}}(r)) ~\txtd  
%W(r) ~\txtd s \right\|_{L^{2}(\Omega,Z)}\\
%& \leq C\int\limits_{0}^{t} (t-s)^{-1} \left( \mathbb{E}\int\limits_{s}^{t} 
%|| \sigma_{v_{1}}(r) - \sigma_{v_{2}}(r)||^{2}_{\mathcal{L}_{2}(H,Z)}
%~\txtd r \right)^{1/2}~\txtd s\\
%& \leq C L_{\sigma} \int\limits_{0}^{t} (t-s)^{-1} \left( \mathbb{E}
%\int\limits_{s}^{t} || v_{1}(r) - v_{2}(r)||^{2}_{Z}~\txtd r
 %\right)^{1/2}~\txtd s \\
%& \leq C L_{\sigma} \sqrt{\widetilde{T}} \left( \mathbb{E} 
%||v_{1} - v_{2} ||^{2}_{\cC([0,\widetilde{T}\wedge\tau_{n}];Z)}
%\right)^{1/2}.
%\end{align*} 
%This further leads to
\begin{align*}
\mathbb{E}\left[|| \widehat{f}^{v_{1}} -\widehat{f}^{v_{2}}
 ||^{2}_{\cC([0,\widetilde{T}\wedge\tau_{n}];Y)}\right]
 \leq CL^{2}_{\sigma}\widetilde{T}^{2(\beta-\alpha)+1} \mathbb{E}\left[||v_{1} 
-v_{2}||^{2}_{ \cC([0,\widetilde{T}\wedge\tau_{n}];Y)}\right].
%\end{align*}
%The same computation in $L^{2}(\Omega,Y)$ gives 
\end{align*}
The last two terms can be estimated as in Lemma \ref{stochconv} applying 
(\ref{growth:sigma}). The result of the computations is
\begin{align*}
\mathbb{E} \left[||\widehat{g}^{u_{1}} - \widehat{g}^{u_{2}} ||^{2}_{
\cC([0,\widetilde{T}\wedge\tau_{n}];Y)} \right]\leq C (C_{R}n l_{
\sigma})^{2} \widetilde{T}^{2 (\nu-\alpha+2\beta) -1} \mathbb{E}\left[||u_{1}-u_{2} 
||^{2}_{\cC([0,\widetilde{T}\wedge\tau_{n}];Y)}\right].
\end{align*}
Finally, another computation very similar to the previous ones gives us 
the following  estimates 
\begin{align*}
\mathbb{E} \left[||\widehat{h}^{u_{1}} - \widehat{h}^{u_{2}} ||^{2}_{
\cC([0,\widetilde{T}\wedge\tau_{n}];Y)} \right]\leq C (C_{R} n l_{\sigma})^{2} 
\widetilde{T}^{2(\nu -\alpha) +1} \mathbb{E}\left[||u_{1}-u_{2}
 ||^{2}_{\cC([0,\widetilde{T}\wedge\tau_{n}];Y)}\right].
\end{align*}
Collecting all these estimates for the terms defined in~\eqref{eq:defah} and
choosing $\widetilde{T}$ small enough proves the statement.\qed
	\end{proof}\\
	
	The following result is the analogue of Lemma \ref{ps1} in the semilinear case.
\begin{lemma}
	Let $0<\varepsilon<1$. Under the assumptions of Theorem \ref{multiplicative} it holds
	\begin{align*}
	\mathbb{P}(\tau_{n}>\varepsilon) > 1- \varepsilon^{2} (C_{1} \mathbb{E}||u_{0}||^{2}_{Z} + C_{2}\varepsilon),
	\end{align*}
	where the positive constants $C_{1}$ and $C_{2}$ are independent of $u_{0}$.
\end{lemma}
\begin{proof}
	The statement can be shown analogously to Lemma \ref{ps1}, compare the proof of Theorem 1.3 in \cite{Kim1}.
	 One can find a positive number $n$ such that   $\frac{1}{n+1}\leq\varepsilon<\frac{1}{n}$ and
	 can use (\ref{u2}) to derive estimates for
	\begin{align*}
	\mathbb{E}\sup\limits_{0\leq t \leq \widetilde{T}\wedge\tau_{n}} ||u(t)||^{2}_{Z},
	\end{align*}
which provide bounds for
\begin{align*}
	\mathbb{E}\sup\limits_{0\leq t \leq \varepsilon} ||u(t\wedge\tau_{n})||^{2}_{Z}.
\end{align*}
Using these as in (\ref{b:stopping}) and regarding the local growth boundedness of $f$ and $\sigma$ specified in (\ref{growth:f}) and (\ref{growth:sigma}), one infers that
\begin{align*}
\mathbb{E} \sup\limits_{0\leq t \leq \varepsilon} ||u(t\wedge\tau_{n})||^{2}_{Z} & \leq C \mathbb{E} ||u_{0}||^{2}_{Z} +\varepsilon^{2(1-\beta)} C (l_{F}(n) + 1 ) \varepsilon \int\limits_{0}^{\varepsilon}\mathbb{E} ||u(s\wedge\tau_{n})||^{2}_{Z} ~\txtd s \\
&+ C\varepsilon (l_{\sigma}(n)+1) \int\limits_{0}^{\varepsilon}\mathbb{E} ||u(s\wedge\tau_{n})||^{2}_{Z}~\txtd s.
\end{align*}
Gronwall's Lemma and Chebyshev's inequality prove the statement as argued in Lemma \ref{ps1}. 
\qed
	\end{proof}

\begin{remark}
Of course, one could also make global Lipschitz assumptions on $f$ and $\sigma$ and thereafter use suitable cut-offs as in the semilinear case or as in \cite{Hornung1}.
Namely one can consider the standard cut-off function $h_{n}:Z\to Z$ defined as
\begin{align*}
h_{n}u:=\begin{cases}
u, & \mbox{if } ||u||_{Z}\leq n \\
\frac{n u}{||u||_{Z}}, & \mbox{if } ||u||_{Z}>n
\end{cases}
\end{align*}
and show that $f_{n}:=h_{n}f$ and $\sigma_{n}:=h_{n}\sigma$ are globally Lipschitz continuous.
Here we have directly localized 
the assumption.
\end{remark}

From all these deliberations we finally conclude

\begin{theorem}
\label{thm:maxlocsol}
There exists a unique maximal local pathwise mild solution 
of (\ref{multiplicative}) $u\in L^{0}(\Omega; \cB([0,\tau_{\infty});Z))\cap L^{0}(\Omega; \cC^{\delta}([0,\tau_{\infty});Y)) $, 
where $\tau_{\infty}:=\lim\limits_{n\uparrow\infty} \tau_{n}$ a.s.
\end{theorem}

\begin{proof}
The proof in \cite[Section 4]{brzezniak} and \cite[Section 3]{BrzezniakHausenblasRazafimandimby} adapts to our setting. We denote by $\mathcal{S}$ the set of all stopping times such that $\tau\in\mathcal{S}$ if and only if there exists a process $(u(t))_{t\in[0,\tau)}$ such that $(u,\tau)$ is the unique local pathwise mild solution of (\ref{multiplicative}). 
 For each $n\in\mathbb{N}$ we take $\tau_{n}$ such that $(u_{n},\tau_{n})$ is the unique local pathwise mild solution of (\ref{multiplicative}). This means that for each $n\in\mathbb{N}$ the pair $(u_{n},\tau_{n})$ is the local mild solution of (\ref{multiplicative}) where $\tau_{n}:=\inf\{t\geq 0 \mbox{ : } ||u_{n}(t)||_{Z}\geq n \}\wedge T$ for some $T>0$.
We now show that $\{\tau_{n},n\in\mathbb{N}\}$ is an increasing 
sequence of stopping times and possesses therefore a limit. This will give us the lifetime of $u$. To this 
aim for $n<m$ let $\tau_{n,m}:=\inf\{t\geq 0\mbox{ : } ||
u_{m}(t)||_{Z}\geq n \}\wedge T$. One can show arguing by contradiction that 
$\tau_{n}<\tau_{m}$ a.s.~if $n<m$. Since $\tau_{n,m}\leq \tau_{m}$ 
a.s.~for $n<m$ we obtain that $(u_{m},\tau_{n,m})$ is a local solution 
of (\ref{multiplicative}) as well as $(u_{n},\tau_{n})$. If $\tau_{n}>
\tau_{n,m}$ a.s. then due to the uniqueness of the pathwise mild 
solution of (\ref{multiplicative}) we infer that $u_{n}(t)=u_{m}(t)$ 
a.s.~for all $t\in[0,\tau_{n}\wedge\tau_{n,m}]=[0,\tau_{n,m}]$. This 
means that $\tau_{n,m}$ is the first exit time for $u_{n}$ with 
$\tau_{n,m}<\tau_{n}$ a.s., which is obviously a contradiction.
Therefore, we conclude that $(\tau_{n})_{n\in\mathbb{N}}$ is an increasing 
sequence of stopping times and possesses the limit $\tau_{\infty}:=\lim
\limits_{n\uparrow\infty} \tau_{n}$ a.s. Let $\{u(t)\mbox{ : } t\in
[0,\tau_{\infty})\}$ be the stochastic process defined by
$$u(t) :=u_{n}(t),\mbox{  for } t\in[\tau_{n-1},\tau_{n}), 
\hspace*{2 mm} n\geq 1,$$
where $\tau_{0}:=0$. Again, due to uniqueness we have that 
$u(t\wedge\tau_{n})=u_{n}(t\wedge\tau_{n})$ for $t>0$. All in all we 
obtained a local pathwise mild solution $(u,\tau_{\infty})$ of (\ref{multiplicative}). The last step is to show 
that this is indeed a maximal local pathwise mild solution. To this aim, 
we infer that a.s.~on the set $\{\omega: \tau_{\infty}(\omega) <T\}$
\begin{align*}
\lim\limits_{t\uparrow\tau_{\infty}} \sup\limits_{0\leq s \leq t } 
|| u(s)||_{Z} & \geq \lim\limits_{n\nearrow\infty} \sup\limits_{0\leq 
s \leq \tau_{n}} ||u(s)||_{Z}= \lim\limits_{n\nearrow\infty} \sup
\limits_{0\leq s \leq \tau_{n}} ||u_{n}(s)||_{Z}=\infty.
\end{align*}
Consequently, $(u,\tau_{\infty})$ is a maximal local pathwise 
mild solution of (\ref{multiplicative}).\qed
\end{proof}

\begin{remark}
	Since
	\begin{align*}
	 \left\{\omega: \sup\limits_{0\leq t \leq \varepsilon} ||u(t\wedge\tau_{n})||_{Z} <n \right\} \subset \{\omega : \tau_{n}(\omega) >\varepsilon\} \subset \{\omega: \tau_{\infty}(\omega)>\varepsilon \},
	\end{align*}
	obviously 
	\begin{align*}
	\mathbb{P}(\tau_{\infty} >\varepsilon )>\mathbb{P}(\tau_{n}>\varepsilon)>0.
	\end{align*}
\end{remark}

\begin{remark} 
In order to show that a solution is global-in-time one, it 
would remain to prove that $\tau_{\infty}=\infty$ a.s. As we would expect and 
as we shall see in Section~\ref{examples}, global-in-time existence can obviously
fail to hold. However, in many applications, additional structure of the 
quasilinear PDE may be enough to also obtain global results. For example, if
the determinisitc PDE part is a cross-diffusion system with an entropy structure 
\cite{Juengel1} and if the noise is multiplicative, we expect that the maximal local pathwise mild solution obtained 
in Theorem \ref{thm:maxlocsol} is indeed a global one. 
We plan to investigate this in a future work using for instance using Khashminski's 
test for non-explosion; see for example~\cite[Lemma 4.1]{brzezniak}, 
\cite[Theorem 3.2]{chow} or \cite[Section 5]{LvDuan}.
\end{remark}

\section{Applications: The Shigesada-Kawasaki-Teramoto Model} 
\label{app}

Let $G\subset \mathbb{R}^{2}$ be an open bounded $\cC^{2}$-domain. We fix 
parameters $k_{1}, k_{2}, \delta_{11},\delta_{21}>0$. We want to study 
a cross-diffusion SPDE, which has been originally introduced by 
Shigesada, Kawasaki and Teramato~\cite{ShigesadaKawasakiTeramoto} in
the deterministic setting in order to analyze population segregation by induced cross-diffusion. Note that the nonlinear term correspond to those arising in the classical Lokta-Volterra competition model. The 
stochastic SKT system is given by
\begin{equation}
\label{skt}
\begin{cases}
~\txtd u=(\Delta (k_{1} u+ a u v + c u^{2})+ \delta_{11} u 
- \gamma_{11} u^{2} - \gamma_{12} uv ) ~\txtd t + \sigma_{1}
(u,v)~\txtd W_{1}(t), & t>0, ~ x\in G\\
~\txtd v=(\Delta (k_{2} v+ b u v + d v^{2})+ \delta_{21} v 
- \gamma_{21} uv - \gamma_{22} v^{2} ) ~\txtd t + \sigma_{2}
(u,v)~\txtd W_{2}(t),& t>0, ~ x\in G \\
\frac{\partial u}{\partial n} =\frac{\partial v}{\partial n} 
=0, & t>0, ~ x\in \partial G,\\
u(x,0)=u_{0}(x)\geq0, \mbox{  }v(x,0)=v_{0}(x)\geq0 & x\in G,
\end{cases}
\end{equation}
where $W_1$, $W_2$ are stochastic processes as defined in (Y4) below. 
Here $u=u(x,t)$ and $v=v(x,t)$ denote the densities of two competing 
species $S_{1}$ and $S_{2}$ in a certain position $x\in G$ at time $t$. 
The coefficients $\gamma_{11},\gamma_{22}>0$ denote the intraspecies 
competition rates in $S_{1}$, respectively in $S_{2}$ and $\gamma_{12},
\gamma_{21}>0$ stand for the interspecies competition rates between $S_{1}$ 
and $S_{2}$. Furthermore, the terms $\Delta(c u ^{2})$ and $\Delta(d v^{2})$ 
represent the self-diffusions of $S_{1}$ and $S_{2}$ with rates $c,d\geq 0$, 
and $\Delta(a uv)$, $\Delta(b uv)$ represent the cross-diffusions of $S_{1}$ 
and $S_{2}$ with rates $a,b\geq 0$. We study the SKT model~(\ref{skt}) in 
its divergence form with linear part $\mbox{div}(\cA(U)\nabla U)$, where 
$$\cA(U)=\begin{pmatrix}
k_{1}+  2 c u+ a v & a u \\
b v & k_{2}+2 d v +b u
\end{pmatrix}
$$
where $U:=(u,v)^\top$. We denote the nonlinear term by
$$
F(U)=\begin{pmatrix}
\delta_{11} u - \gamma_{11} u^{2} -\gamma_{12}uv\\
\delta_{21} v - \delta_{21} uv -\gamma_{22} v^{2}
\end{pmatrix}.$$
We assume here that the parameters are chosen so that $\cA(U)$ is positive 
definite. For applications, the most interesting case occurs under the restriction
that $u,v\geq 0$ should be preserved. If we would know this, then it
suffices to impose $$ a^{2}<8 c b \quad \mbox{and}\quad  b^{2} <8 da ,$$
which is a necessary and sufficient condition for positive definiteness of 
$\cA(u)$. One can replace this 
by the even weaker condition $ab<64 cd$~\cite[Chapter 15, Section 3]{Yagi1}.\medskip

Our aim is to formulate equation (\ref{skt}) as an abstract quasilinear SPDE, 
as investigated in Section \ref{qspde}, on $X:=\mathbb{L}^{2}(G)= 
L^{2}(G)\times L^{2}(G)$. Throughout this section we use the same notations as in Section \ref{qspde}.
 We set $Z:=\mathbb{H}^{1+\varepsilon}(G)= 
H^{1+\varepsilon}(G)\times H^{1+\varepsilon}(G)$, for a fixed 
$0<\varepsilon<1/2$ and rewrite (\ref{skt}) as
\begin{equation}
\label{cp}
\begin{cases}
~\txtd U(t) =( A(U(t)) U(t) + F(t,U(t)))~\txtd t + \sigma(t,U(t))~\txtd
\mathbb{W}(t), \hspace*{3 mm} t\in[0,T].\\
U(0)=U_{0} \in K \mbox{ a.s.},
\end{cases}
\end{equation}
 where $\mathbb{W}:=(W_1,W_2)^\top$. 
According \cite[Proposition 15.1]{Yagi1}, there is a sectorial 
operator $A(U)$, defined via the matrix $\cA(u)$ in a standard way~\cite{Yagi1},
of angle $0<\varphi<\frac{\pi}{2}$ for $U\in \mathcal{U}_{T}$, so 
we are justified to introduce $X_{\widehat\mu}:=D(A(U)^{\widehat\mu})$, for $\widehat\mu\geq 0$, see 
below.\\

The following assertions regarding the deterministic part of (\ref{skt}) 
are stated and proved in \cite[Chapter 5]{Yagi1} and \cite[Section 3]{Yagi}. Therefore all the assumptions made in the previous section are satisfied for this example.

\begin{remark}
For more general deterministic quasilinear problems and assumptions on the 
coefficients for which the next statements hold true, see \cite[Section 10]{Amann2}.
\end{remark}

\begin{itemize}
	\item [(Y1)] For $U\in \cU_{T}$ due 
	to~\cite[Proposition 15.2]{Yagi1} $D(A(U))=H^{2}_{N}(G)\times H^{2}_{N}
	(G):=\mathbb{H}^{2}_{N}(G)$ and due to~\cite[Proposition 15.3]{Yagi1}
	\begin{align}
	X_{\widehat\mu} =\mathbb{H}^{2\widehat\mu}(G), ~ \mbox{ for } & 0\leq \widehat\mu <\frac{3}{4}\\
	X_{\widehat\mu} =\mathbb{H}^{2\widehat\mu}_{N}(G), ~ \mbox{ for } & \frac{3}{4}<\widehat\mu\leq 1.
	\end{align}
	In this context we infer that $Z=D(A(U)^{\beta})$ for $\beta=1/2+\varepsilon/2$.
	\item [(Y2)] According to \cite[(15.10), p. 492]{Yagi1}, the following 
	local Lipschitz continuity of the generators holds true: there exist a constant $\widetilde{L}=\widetilde{L}(U,V)>0$ such that
	\begin{equation}\label{lipa}
	|| A(U)-A(V)||_{\mathcal{L}(\mathbb{H}^{2}_{N}(G), X)}\leq 
	\widetilde{L}(U,V)||U-V ||_{Y},\mbox{ for } U,V\in\mathcal{U}_{T}.
	\end{equation}
Here $Y:=D(A(U)^{\alpha})$ for $0\leq \alpha\leq\frac{1+\varepsilon_{0}}{2}$, where $0<\varepsilon_{0}<\varepsilon$, see also \cite{Yagi}.
	\item [(Y3)] Let $F: \Omega\times \mathcal{U}_{T}\to X$. There exist 
	constants $L_{F}=L_{F}(u), l_{F}=l_{F}(u)>0$ such that
	$$||F(U)-F(V)||_{X}\leq L_{F} || U -V ||_{X}, \mbox{ for } 
	U,V\in \mathcal{U}_{T}$$
	and $$ || F(U)||_{X} \leq l_{F} (1+ ||U ||_{X}) \mbox{, for } 
	U\in \mathcal{U}_{T}. $$
	The local Lipschitz continuity is satisfied since $F$ is a square function 
	of $u$ and $v$, see~\cite{Yagi} and \cite{Yagi1}. %%For the Lipschitz 
%%	continuity of $F$ as a mapping onto $X_{\widehat\mu}$ as discussed in Section \ref{qspde}, 
%%	consult \cite[Lemma 10.1, (iii)]{Amann2} and the references specified therein.
	\item [(Y4)] $\mathbb{W}=(W_{1}, W_{2})$ is an $H$-cylindrical Brownian 
	motion, $\sigma:\Omega\times \mathcal{U}_{T}\to 
	\mathcal{L}_{2}(H,X_{2\beta})$, where $H$ stands for a separable Hilbert 
	space. Furthermore, there exist constants $L_{\sigma}=L_{\sigma}(u,v), l_{\sigma}=l_{\sigma}(u)>0$ such that
 	$$||\sigma(U)-\sigma(V) ||_{\mathcal{L}_{2}(H,X_{2\beta})} \leq 
	L_{\sigma} ||U -V||_{X}, \mbox{ for } U,V \in \mathcal{U}_{T} ;$$
	respectively
	\begin{equation} 
	\label{gb}
	||\sigma(U)||_{\mathcal{L}_{2}(H,X_{2\beta})} \leq l_{\sigma}
	(1+|| U||_{X}), \mbox{ for } U\in \mathcal{U}_{T} . 
	\end{equation}
\end{itemize}

Keeping this in mind, we conclude that all assumnptions made in Section \ref{qspde} are fulfilled and the restrictions on the exponents $\alpha$, $\beta$ and $\nu$ imposed in (A4') hold.
\begin{remark} 
Note that non-negativity of local solutions for (\ref{skt}) is not ensured 
by (\ref{gb}). There is actually a trade-off: if we allow for additive noise 
as in (\ref{gb}), then we need a very strong assumption of uniform positive
definiteness for $\cA(u)$ but if we allow for more general matrices 
$\cA(u)$, then we need more assumptions on the noise, e.g., we \emph{conjecture}
that the assumption
\begin{equation}
		||\sigma(U)||_{\mathcal{L}_{2}(H,X_{2\beta})} \leq l_{\sigma}|| U||_{X}, 
		\mbox{ for } U\in \mathcal{U}_{T}, 
\end{equation}
together with $u_0>0$, $v_0>0$ uniformly in space, will imply short-time existence
up to a stopping time and preserve positivity, see ~\cite{Assing} or Theorem 1.3 in \cite{Kim1}. 
\end{remark}

Regarding the assumptions (Y1)-(Y4), we infer that in the context of 
Section~\ref{qspde} we have $Z=X_{\beta}$ for $\beta=\frac{1+\varepsilon}{2}$ 
and $Y=X_{\alpha}$ for $\alpha$ specified in \ref{lipa} . Note that 
(\ref{lipa}) implies that (\ref{at}) and (\ref{nu}) are fulfilled with $\nu=1$.
Therefore, we apply for (\ref{skt}) the abstract results proved in 
Section~\ref{qspde} and infer:

\begin{theorem}
Under the assumptions stated in this section the stochastic SKT equation 
(\ref{skt}) possesses a unique  maximal local pathwise mild solution 
$U\in L^{0}(\Omega; \cB([0,\tau_{\infty});Z)) 
\cap L^{0}(\Omega; \cC^{\delta}([0,\tau_{\infty});Y)) $.
\end{theorem}

\section{Examples}
\label{examples}

As already known and well-established in the deterministic case, quasilinear 
PDEs do not possess global solutions without further assumptions. The aim of 
this subsection is to present simple examples of stochastic PDEs with 
cross-diffusion so that their solution cannot exist globally.\medskip 

\subsection{A Cross-Diffusion SPDE with Finite-Time Blow-Up}

In this setting we give 
an example of a cross-diffusion equation that exhibits finite-time blow up 
in the deterministic case and prove that this holds true also in the stochastic 
one. This fact is not surprising, since we consider here only white 
noise~\cite[Theorems 4.1-4.3]{LvDuan} but it seems useful to have for completeness.
To this aim, we denote by $\phi$ the normalized eigenfunction associated to 
the first eigenvalue $\lambda_{1}$ of the Dirichlet-Laplacian in $G$, where 
$G\subset\mathbb{R}^{n}$ is an open-bounded $\cC^{2}$-domain. More precisely, we 
have that
\begin{equation}\label{phi}
\begin{cases}
& \Delta \phi = -\lambda _{1} \phi, \mbox{ in } G; \hspace*{2 mm} 
\phi=0,  ~ x\in\partial G;\\
& \int\limits_{G} \phi(x)~\txtd x=1.
\end{cases}
\end{equation}
Note that $\phi(x)>0$ for $x\in G$. Keeping this in mind, we consider 
the following SPDE: 

\begin{example}
	\begin{equation}\label{blowupinf}
		\begin{cases}
	~\txtd u= (\Delta u + \frac{1}{2} \Delta v +u^{2}+\frac{
	\lambda_{1}}{2} v) ~\txtd t + \sigma(u)~\txtd W(t), & t>0, ~ x\in G\\
	~\txtd v= (\Delta v + (\lambda_{1} + k ) v)~\txtd t, & t>0, ~ x\in G ,\\
	u|_{\partial G} = v|_{\partial G} = 0, & t>0,\\
	u(x,0)=u_{0}(x)\geq 0, \mbox{ } v(x,0)=\phi(x), &x\in G.
	\end{cases}
	\end{equation}
\end{example}

In this case we set $U:=(u,v)^{T}$ and have
$$A(U)=\begin{pmatrix}
1 & \frac{1}{2} \\
0 & 1
\end{pmatrix}\Delta
\quad \mbox{ and } \quad 
F(U)=\begin{pmatrix}
u^{2} + \frac{\lambda_{1}}{2} v\\
(\lambda_{1}+ k) v
\end{pmatrix}
.$$
The only requirement for $\sigma$ is  
local-Lipschitz continuity in order to ensure the existence of a local 
solution for (\ref{blowupinf}). As we will see in the next computation 
the stochastic term will not play a role due to the fact that the 
expectation of the It\^{o}-integral is $0$.
In the deterministic case we know that the first component of the 
solution of (\ref{blowupinf}) blows up in finite-time. We now show 
that this remains valid in the stochastic setting.

\begin{lemma}
\label{ex1}
Consider the SPDE~\eqref{blowupinf}. 
There exists $u_0\geq0$ and a finite time $T^{*}$ such that 
\begin{equation}
\lim\limits_{t\nearrow T^{*}} \mathbb{E}\left[ \sup\limits_{x\in G} u(x,t)\right]
=+\infty.
\end{equation}
\end{lemma}

\begin{proof} 
Since $v(x,t)=\txte^{kt} \phi(x)$ is the solution of the second equation, 
the first one results in 
\begin{equation}
	\label{firstc}
	\begin{cases}
	~\txtd u = (\Delta u + u ^{2} ) ~\txtd t + \sigma(u)~\txtd W(t)\\
	u(0)=u_{0}(x)\geq 0.
	\end{cases}
\end{equation}
Note that (\ref{firstc}) is a parabolic SPDE. Under the above assumptions, 
it is known that this possesses a local positive solution which exhibits 
finite-time blow-up \cite[Theorem 4.1]{LvDuan}. For the convenience of the 
reader we indicate the proof of this statement. Assuming that there exists 
a global solution $u$ of (\ref{firstc}) such that for any $T>0$
\begin{equation}
\label{contradiction}
	\sup\limits_{0\leq t \leq T} \mathbb{E} \left[\sup\limits_{x\in G} u(x,t)\right]<\infty.
	\end{equation}
we immediately also get that
	\begin{equation*}
	\sup\limits_{0\leq t \leq T} \mathbb{E}\left[ \int\limits_{G} 
	u(x,t)\phi(x)~\txtd x\right] \leq \sup\limits_{0\leq t \leq T} \mathbb{E}\left[ 
	\sup\limits_{x\in G } u(x,t)\right]<\infty.
\end{equation*}
We set 
$$y(t):=\int\limits_{G} u(x,t)\phi(x)~\txtd x \quad \text{for $t\geq 0$},$$ 
multiply~\eqref{firstc} by $\phi$, take the expectation and obtain via a 
direct application of integration-by-parts and Fubini's Theorem that
\begin{align*}
	\mathbb{E}\left[ y(t) \right]= \underbrace{\int\limits_{G} u_0(x)\phi(x)~\txtd 
	x}_{=:(u_0,\phi)} -\lambda _{1} 
\int\limits_{0}^{t} \mathbb{E} \left[y(s)\right]~\txtd s + \int\limits_{0}^{t} 
\mathbb{E} \left[\int\limits_{G} u^{2}(x,s)\phi(x)~\txtd x~\txtd s\right].
\end{align*}
Setting $\widetilde{y}(t):=\mathbb{E} \left[y(t)\right]$ for $t>0$ and 
differentiating with respect to $t$ we obtain
\begin{align*}
\begin{cases}
\frac{\txtd\widetilde{y}(t)}{\txtd t}=-\lambda_{1} \widetilde{y}(t) 
+ \mathbb{E} \left[\int\limits_{G}u^{2}(x,t)\phi(x)\right]~\txtd x \\
\widetilde{y}(0)=(u_{0},\phi).
\end{cases}
\end{align*}
From Jensen's and Cauchy-Schwarz inequality we have
\begin{align*}
\widetilde{y}^{2}(t)&=\mathbb{E} \left[\int\limits_{G} u(x,t)
\phi(x)~\txtd x \right]^{2} \leq \mathbb{E} \left[  \int\limits_{G}
 u(x,t)\phi(x)~\txtd x\right]^{2} \\
&\leq  \mathbb{E}\left[ \int\limits_{G} u^{2}(x,t) \phi(x)~\txtd x \right] 
\int\limits_{G}\phi(x)~\txtd x,
\end{align*}
Consequently, using (\ref{phi})
\begin{align}\label{b}
\begin{cases}
\frac{d\widetilde{y}(t)}{dt}\geq -\lambda_{1} \widetilde{y}(t) 
+ \widetilde{y}^{2}(t)\\
\widetilde{y}(0)=(u_{0},\phi).
\end{cases}
\end{align}
So $\widetilde{y}$ must blow up in a finite time for a suitable 
$u_0$, which is a contradiction to (\ref{contradiction}). This proves 
the assertion.\qed
\end{proof}

\subsection{A Cross-Diffusion SPDE with Degenerating Quasilinear Operator}

We first construct an example, in which the solution blows up in 
finite time and cannot remain positive starting with a positive 
initial condition. Therefore, in this case the maximum principle 
is not valid. To this aim, letting $k>\lambda_{1}$ as in Lemma~\ref{ex1}, 
we consider
\begin{example}
	\begin{equation}\label{blowup}
	\begin{cases}
	~\txtd u= (\Delta u + \frac{1}{2} \Delta v + 
	u (2\lambda_{1} -u)) ~\txtd t + \sigma(u)~\txtd W(t), &t>0, ~ x\in G,\\
	~\txtd v=( \Delta v + (\lambda_{1} + k ) v)~\txtd t, & t>0, ~ x\in G,\\
	u|_{\partial G} = v|_{\partial G} = 0, & t>0,\\
	u(x,0)=u_{0}(x) \geq 0, \mbox{ } v(x,0)=\phi(x), & x\in G.
	\end{cases}
	\end{equation}
\end{example}
Here we have for $U=(u,v)^{T}$ 
$$A(U)=\begin{pmatrix}
1 & \frac{1}{2} \\
0 & 1
\end{pmatrix}\Delta
\quad \mbox{ and } \quad  
F(U)=\begin{pmatrix}
u(2\lambda_{1} - u)\\
(\lambda_{1}+ k) v
\end{pmatrix}
.$$

In the deterministic case it is known that $u$ blows up in finite 
time and cannot remain positive~\cite[Theorem 1.6]{le}. We investigate 
now this situation by similar methods as in Lemma \ref{ex1} in the stochastic 
framework.

\begin{lemma}
Consider the SPDE~\eqref{blowup}. There exists a finite time $T^{*}$ such that 
 \begin{equation}\label{bu}
 \lim\limits_{t\nearrow T^{*}-} \mathbb{E}\left[ \sup\limits_{x\in G}u(x,t)\right]
=-\infty.
 \end{equation}
\end{lemma}
\begin{proof} 
Since $v(x,t)=\txte^{kt} \phi(x)$ is the solution of the second equation, 
the first one results in 
\begin{equation}
\label{firstcomponent}
	~\txtd u= (\Delta u -\frac{\lambda_{1}}{2}\txte^{k t}\phi + 
	2\lambda_{1}u- u ^{2} ) ~\txtd t + \sigma(u)~\txtd W(t).
\end{equation}
We prove the assertion by similar methods to those presented in Example~\ref{blowupinf}
by setting $$y(t):=\int\limits_{G} u(x,t)\phi(x)~\txtd x$$ 
and by also defining $\widetilde{y}(t):=\mathbb{E} \left[y(t)\right]$. 
A direct calculation yields
\begin{equation*}
\begin{cases}
\frac{d\widetilde{y}(t)}{dt} = \lambda _{1} \widetilde{y}(t) - 
\frac{\lambda_{1}}{2} \txte^{kt}  \int\limits_{G} \phi^{2}(x)dx - 
\mathbb{E} \int\limits_{G} u^{2}(s,x)\phi(x)dx\\
\widetilde{y}(0)=(u_{0},\phi).
\end{cases}
\end{equation*}
Again, we infer due to Jensen's inequality that
\begin{equation*}
\begin{cases}
\frac{d\widetilde{y}(t)}{dt} \leq \lambda _{1} \widetilde{y}(t) 
- \frac{\lambda_{1}}{2} \txte^{kt}  \int\limits_{G} \phi^{2}(x)~\txtd x 
- \widetilde{y}^{2}(t)\\
\widetilde{y}(0)=(u_{0},\phi).
\end{cases}
\end{equation*}
In order to reach a contradiction, we combine the inequalities
$$\frac{\txtd\widetilde{y}(t)}{\txtd t} \leq \lambda _{1} \widetilde{y}(t) 
- \frac{\lambda_{1}}{2} \txte^{kt}  \int\limits_{G} \phi^{2}(x)~\txtd x\quad 
\mbox{  and  }  \frac{\txtd\widetilde{y}(t)}{\txtd t} \leq \lambda _{1} 
\widetilde{y}(t) - \widetilde{y}^{2}(t). $$
The first one entails using Gronwall's Lemma
\begin{equation}
\label{gronwall}
\widetilde{y}(t) \leq \txte^{\lambda_{1} t} \left(  (u_{0},\phi) 
-\frac{\lambda_{1}(\txte^{(k-\lambda_{1})t}-1 ) }{2(k-\lambda_{1})} 
\int\limits_{G} \phi^{2}(x)~\txtd x  \right), \mbox{  for } t>0.
\end{equation}
Since $k>\lambda_{1}$, the estimate (\ref{gronwall}) implies that 
there exists $t_{1}>0$ such that $\widetilde{y}(t)<0$ for $t\geq t_{1}$, also
$$ \mathbb{E} \left[\int\limits_{G} u(t,x)\phi(x)~\txtd x\right]<0, \quad
\mbox{ for } t\geq t_{1}.$$
Consequently, for $t\geq t_{1}$, the inequality
$
\frac{\txtd\widetilde{y}(t)}{\txtd t} \leq \lambda_{1} \widetilde{y}(t) 
- \widetilde{y}(t)^2
$
implies that
\begin{equation*}
\frac{\txtd\widetilde{y}(t)}{\txtd t}\frac{1}{\widetilde{y}^{2}(t)} - \lambda_{1} 
\frac{1}{\widetilde{y}(t)} \leq -1.
\end{equation*}
Setting $w:=\frac{-1}{\widetilde{y}}$, we obtain for $t\geq t_{1}$ 
that $w(t)>0$ and $$\frac{\txtd w(t)}{\txtd t} + \lambda_{1} w(t)\leq -1.$$ 
It is now elementary to check that $w(t)\to 0$ as $t\nearrow T^{*}$ for 
a finite time $T^*$. This implies $\widetilde{y}(t)\to -\infty$ as $t\nearrow T^{*}-$
proving the claim.\qed
\end{proof}

To conclude, regarding Example~\ref{blowup} one can now easily construct a
cross-diffusion quasilinear SPDE, which becomes ill-posed. Indeed, we just have 
 in a third component an equation degenerating into the backward heat-equation.

\begin{example}
		\begin{equation}
		\label{degenerate}
	\begin{cases}
	~\txtd u= (\Delta u + \frac{1}{2} \Delta v + u (2\lambda_{1} -u)) 
	~\txtd t + \sigma(u)~\txtd W(t), & t>0, ~ x\in G,\\
	~\txtd v=( \Delta v + (\lambda_{1} + k ) v)~\txtd t, & t>0, ~ x\in G,\\
	~\txtd w = u \Delta w ~\txtd t, & t>0, ~ x\in G,\\
	u|_{\partial G} = v|_{\partial G} = 0=w|_{\partial G}=0, &  t>0,\\
	u(x,0)=u_{0}(x) \geq 0, \mbox{ } v(x,0)=\phi(x), & x\in G.
	\end{cases}
	\end{equation}
\end{example}

\appendix

\section{An integration by parts formula}
\label{a}

We shortly indicate the computation for the pathwise mild solutution 
of the linear SPDE
\begin{equation}\label{lin}
	 \txtd U(t)=A(t)U(t)~\txtd t + G(t)~\txtd W(t),
\end{equation}
obtained using an integration by parts formula. 
 
	The strong solution of (\ref{lin})
	\begin{equation}
	U(t)=\int\limits_{0}^{t}A(s) U(s) ~\txtd s + \int\limits_{0}^{t} G(s)~\txtd W(s)
	\end{equation}
	can be written as \begin{equation}\label{u}
	U(t)= S(t,0)\int\limits_{0}^{t} G(s)~\txtd W(s) - \int\limits_{0}^{t} 
	S(t,s)A(s)\int\limits_{s}^{t} G(r) ~\txtd W(r) ~\txtd s,
	\end{equation}	
	where $S(\cdot,\cdot)$ is the evolution operator generated by $A(\cdot)$. 
	If $A(\cdot)$ is bounded, a straightforward computation 
	\cite[Section 4.2, p. 18]{PronkVeraar} immediately proves the 
	claim. Namely, from (\ref{u}) we have
	\begin{align*}
	U(t)& =S(t,0)\int\limits_{0}^{t}G(s) ~\txtd W(s) - \int\limits_{0}^{t}S(t,s)A(s)\int\limits_{0}^{t}G(r) ~\txtd W(r) ~\txtd s \\
	& + \int\limits_{0}^{t} S(t,s)A(s)\int\limits_{0}^{s} G (r) ~\txtd W(r) ~\txtd s.
	\end{align*}
	Since $$\frac{\partial}{\partial s} S(t,s)=-S(t,s)A(s) $$
	we have $$\int\limits_{0}^{t}S(t,s)A(s)~\txtd s=-S(t,t)+S(t,0),$$
	so 
	\begin{align*}
	U(t)& =S(t,0)\int\limits_{0}^{t} G(s)~\txtd W(s) + S(t,t)\int\limits_{0}^{t}G(s)~\txtd W(s)- S(t,0)\int\limits_{0}^{t} G(s) ~\txtd W(s)\\
	& + \int\limits_{0}^{t} S(t,s)A(s)\int\limits_{0}^{s}G(r)~\txtd W(r) ~\txtd s\\
	& = \int\limits_{0}^{t} G(s)~\txtd W(s) +  \int\limits_{0}^{t} S(t,s)A(s)\int\limits_{0}^{s}G(r) ~\txtd W(r) ~\txtd s.
	\end{align*} 
	Furthermore, using Fubini's theorem
	\begin{align*}
	\int\limits_{0}^{t} A(r) U(r)dr & =\int\limits_{0}^{t} A(r) \int\limits_{0}^{r} G(s) ~\txtd W(s) ~\txtd r + \int\limits_{0}^{t}\int\limits_{0}^{r} A(r)S(t,s)A(s) \int\limits_{0}^{s} G(\tau) ~\txtd W(\tau) ~\txtd s  ~\txtd r\\
	& = \int\limits_{0}^{t} A(r) \int\limits_{0}^{r} G(s) ~\txtd W(s) ~\txtd r + \int\limits_{0}^{t} \int\limits_{s}^{t} A(r) S(r,s)A(s)\int\limits_{0}^{s} G(\tau)~\txtd W(\tau) ~\txtd r ~\txtd s\\
	& = \int\limits_{0}^{t} A(s) \int\limits_{0}^{s} G(r)~\txtd W(r)~\txtd s + \int\limits_{0}^{t}S(t,s)A(s)\int\limits_{0}^{s}G(r)~\txtd W(r) ~\txtd s\\ &-\int\limits_{0}^{t}A(s)\int\limits_{0}^{s}G(r)~\txtd W(r)~\txtd s\\
	&=\int\limits_{0}^{t}S(t,s) A(s) \int\limits_{0}^{s} G(r)~\txtd W(r) ~\txtd s.
	\end{align*}
	For the last part we used that $$\frac{\partial}{\partial t} S(t,s)=A(t)S(t,s), $$
	so $$\int\limits_{s}^{t}A(r)S(r,s)~\txtd r=S(t,s)-S(s,s). $$
	If $A(\cdot)$ is unbounded, one can repeat the previous computation under suitable assumptions (as in Section \ref{qspde}) which ensure the existence of the integrals above.

%%%%%%%%%%%%%%%%%%%%%%%%%%%%%%%%%%%%%%%%%%%%%%%%%%%%%%%%%%%%%%%%%%%%%%%%%%%%%%%%%%%%%%%%%%%%%%

%\bibliographystyle{plain}
%\bibliography{../my_refs}

\end{document}